\documentclass[reqno,11pt]{amsproc}
\usepackage{amssymb}
\usepackage{amsmath}
\usepackage{amsthm}
\usepackage{amsfonts}
\usepackage{microtype}
\usepackage[canadian]{babel}
\usepackage{xcolor}
\usepackage{geometry}
\usepackage{tikz-cd}
\usepackage{enumitem}
\usepackage{ifthen}
\usepackage{hyphenat}
\usepackage{mathtools}
\usepackage{cancel}
\usepackage{changepage}
\usepackage{thmtools}

\definecolor{myurlcolor}{rgb}{0,0,0.3}
\definecolor{mycitecolor}{rgb}{0,0.3,0}
\definecolor{myrefcolor}{rgb}{0.3,0,0}
\usepackage[pagebackref,draft=false]{hyperref}
\hypersetup{colorlinks,
linkcolor=myrefcolor,
citecolor=mycitecolor,
urlcolor=myurlcolor}

\usepackage[capitalize]{cleveref}
\usepackage{bookmark}

\newcommand{\beq}{\begin{equation}}
\newcommand{\eeq}{\end{equation}}
\newcommand{\N}{\mathbb{N}}

\newcommand{\Nplus}{\mathbb{N}_{> 0}}			
\newcommand{\Z}{\mathbb{Z}}
\newcommand{\Q}{\mathbb{Q}}
\newcommand{\Qplus}{\mathbb{Q}_{> 0}}			
\newcommand{\R}{\mathbb{R}}
\newcommand{\Rplus}{\mathbb{R}_{> 0}}			
\newcommand{\TR}{\mathbb{TR}}

\newcommand{\B}{\mathbb{B}}				

\newcommand{\op}{\mathrm{op}}
\newcommand{\eps}{\varepsilon}

\newcommand{\Sper}[1]{\mathsf{TSper}(#1)}		

\newcommand{\lc}{\mathrm{lc}}			
\newcommand{\lev}{\mathrm{lev}}			
\DeclareMathOperator{\Frac}{\mathsf{Frac}}	
\DeclareMathOperator{\Newton}{\mathsf{Newton}}	

\swapnumbers
\newtheorem{dummy}{Dummy}[section]
\newtheorem{thm}[dummy]{Theorem}\Crefname{thm}{Theorem}{Theorems}
\newtheorem{lem}[dummy]{Lemma}\Crefname{lem}{Lemma}{Lemmas}
\newtheorem{prop}[dummy]{Proposition}\Crefname{prop}{Proposition}{Propositions}
\newtheorem{cor}[dummy]{Corollary}\Crefname{cor}{Corollary}{Corollaries}
\Crefname{conj}{Conjecture}{Conjectures}
\Crefname{qstn}{Question}{Questions}
\newtheorem{defn}[dummy]{Definition}\Crefname{defn}{Definition}{Definitions}
\Crefname{prob}{Problem}{Problems}
\Crefname{nota}{Notation}{Notations}
\theoremstyle{remark}
\newtheorem{ex}[dummy]{Example}\Crefname{ex}{Example}{Examples}
\newtheorem{rem}[dummy]{Remark}\Crefname{rem}{Remark}{Remarks}
\Crefname{note}{Note}{Notes}

\numberwithin{equation}{section}
\Crefname{enumi}{}{}

\allowdisplaybreaks

\usepackage{enumitem}
\setlist[enumerate]{label=(\alph*),itemsep=5pt,topsep=8pt}
\setlist[itemize]{label=$\triangleright$,itemsep=5pt,topsep=6pt}

\Crefformat{enumi}{#2#1#3}

\let\originalleft\left
\let\originalright\right
\renewcommand{\left}{\mathopen{}\mathclose\bgroup\originalleft}
\renewcommand{\right}{\aftergroup\egroup\originalright}

\usepackage[explicit]{titlesec}

\titleformat{\section}[block]{\bfseries\large\filcenter}{\thesection.}{6pt}{#1}
\titlespacing{\section}{0pt}{18pt}{12pt}
\titleformat{\subsection}[runin]{\bfseries}{\noindent}{0.4em}{#1.}

\usepackage{titletoc}
\titlecontents{part}[0em]
	{\vspace{1pc}}
        {\bfseries\normalsize\contentslabel[\thecontentslabel]{2em}}
	{\bfseries\underline}
	{}
\titlecontents{section}[2em]
	{\vspace{0pt}}
        {\normalfont\normalsize\contentslabel[\thecontentslabel]{2em}}
	{}
	{\titlerule*[.75em]{}\contentspage}

\newcommand{\newterm}[1]{\textbf{#1}}

\title[Abstract Vergleichsstellens\"atze II]{Abstract Vergleichsstellens\"atze\\ for preordered semifields and semirings II}

\author{Tobias Fritz}

\address{Department of Mathematics, University of Innsbruck}
\email{tobias.fritz@uibk.ac.at}



\thanks{\textit{Acknowledgements.}
	We thank Erkka Theodor Haapasalo, Zhang Jiang, Mil\'an Mosonyi, Xiaosheng Mu, Tim Netzer, Luciano Pomatto, Markus Schweighofer, Omer Tamuz, Marco Tomamichel, Becca Verghese, Frits Verhagen, P\'eter Vrana as well as an anonymous referee for useful feedback, discussion and suggestions.}

\begin{document}

\begin{abstract}
	The present paper continues our foundational work on real algebra with preordered commutative semifields and semirings.
	We prove two abstract Vergleichsstellens\"atze for preordered commutative semirings of polynomial growth.
	These generalize the results of Part I by no longer assuming $1 \ge 0$.
	Such a generalization comes with substantial technical complications: our Vergleichsstellens\"atze now also need to take into account infinitesimal information encoded in the form of monotone derivations in addition to the monotone homomorphisms to the nonnegative reals and tropical reals.
	The auxiliary technical results we develop along the way include surprising implications between inequalities in preordered semifields and a type classification for multiplicatively Archimedean fully preordered semifields.
	
	Among other applications, two companion papers use these results in order to derive new results in probability and information theory; one on asymptotics of random walks on topological abelian groups, and the other on the asymptotics of matrix majorization.
\end{abstract}

\maketitle

\tableofcontents


\section{Introduction}

This paper is part of an emerging research program on real algebra with commutative preordered semirings and semifields.
Here, semirings and semifields are like rings and fields, but without the assumption of additive inverses.
In Part I~\cite{partI}, detailed motivation for this project was given. It comes in the form of two main points:
\begin{itemize}
	\item Our Vergleichsstellens\"atze have entirely new applications which are not covered by the classical Positivstellens\"atze.
		For example, take any class of mathematical structures---such as representations of a Lie group---for which notions of direct sum and tensor product exist, and the tensor products distribute over direct sums. Then the isomorphism classes form a semiring. In many cases, this semiring carries a canonical preorder with respect to one structure being included (up to isomorphism) in another.
		One then obtains a preordered semiring, and results such as our Vergleichsstellens\"atze are useful tools in understanding its structure.

		The results of the present paper have found their first applications in probability and information theory.
		This comprises new formulas for the asymptotics of random walks on topological abelian groups~\cite{arw} as well as for the asymptotics of matrix majorization~\cite{major}.
	\item From the perspective of real algebra itself, semifields have a number of advantages over fields that make them of intrinsic interest.
		One such advantage is that if $\R_+(X)$ is the semifield of rational functions with nonnegative coefficients, then the evaluation maps are well-defined homomorphisms $\R_+(X) \to \R_+$, in stark contrast to the lack of evaluation homomorphisms on fields of rational functions.
		Another advantage is that semifields can combine the convenience of multiplicative inverses with the presence of nilpotency. For example, there is a semifield $F \coloneqq \R_{(+)}[X] / (X^2)$ whose nonzero elements are the linear polynomials $r + s X$ with $r > 0$. Since its enveloping ring is the ring of dual numbers, $F \otimes \Z \cong \R[X] / (X^2)$, it is clear that $F$ cannot be embedded into a field.
\end{itemize}
As in real algebra generally, to understand the structure of a particular preordered semiring, it is useful to probe this structure through homomorphisms to test objects such as the real numbers. These homomorphisms can be thought of in geometrical terms as points of a spectrum.
By analogy with Positivstellens\"atze, Part I has introduced the term \newterm{Vergleichsstellensatz} for a type of result which relates the given algebraic preorder to a spectral preorder.
It provides sufficient conditions for the existence of an algebraic certificate witnessing a spectral preorder relation.

More concretely, in Part I we first proved a separation theorem for preordered semifields, which states that every semifield preorder is the intersection of its total semifield preorder extensions.
From this, we derived a Vergleichsstellensatz for a certain class of preordered semirings $S$ in which $1 > 0$.
In this result, the role of the spectrum is played by the monotone homomorphisms $\phi : S \to \R_+$ and $\phi : S \to \TR_+$, where $\TR_+ \coloneqq ([0,\infty),\max,\cdot)$ is the semifield of tropical reals.
Roughly speaking, for nonzero $x, y \in S$ our Vergleichsstellensatz 
considers two kinds of algebraic certificates:
\begin{itemize}
	\item A \newterm{catalytic} certificate, stating that there is nonzero $a \in S$ with
		\beq
			\label{cat_cert}
			a x \le a y.
		\eeq
	\item An \newterm{asymptotic} certificate, which states that
		\beq
			\label{asymp_cert}
			x^n \le y^n \qquad \forall n \gg 1.
		\eeq
\end{itemize}
These are certificates for the spectral preorder relation in the sense that if either of them holds, then $\phi(x) \le \phi(y)$ for all $\phi$ is easily implied.
Our Vergleichsstellensatz now provides an almost converse: if we have strict inequality $\phi(x) < \phi(y)$ for all such $\phi$, then both algebraic certificates hold.
The fully formal statement with precise assumptions will be recalled as \Cref{Imain}.

The goal of the present paper is to dig deeper and develop abstract Vergleichsstellens\"atze that apply even if $1 \not\ge 0$.
This may seem like an artificial problem on first look, because why would one want to do this?
However, there actually are many applications, for example to probability and information theory, where preordered semirings with $1 \not\ge 0$ appear.
To see why, it is enough to note that probability measures are by definition normalized to $1$, which indicates that one only wants measures of the same normalization to be comparable at all.
In particular, the zero measure will not be comparable to any normalized measure, resulting in $1 \not\ge 0$ and $1 \not\le 0$.
Two such applications of the results of this paper have been worked out in companion papers~\cite{arw,major}, and we now give a brief sketch of what the relevant preordered semiring is\footnote{We do so in a watered-down version where only finitely supported measures are considered; the preordered semiring actually considered in~\cite{arw} has arbitrary measures (of compact support) on a topological abelian group as elements.} and why one has $1 \not\ge 0$.
The elements of the polynomial semiring $\R_+[X_1,\ldots,X_d]$ can be identified with finitely supported measures on $\N^d$, where addition of polynomials corresponds to addition of measures and multiplication of polynomials corresponds to convolution of measures. The semiring preorder generated by
\[
	X_1 \ge 1, \quad \ldots, \quad X_d \ge 1
\]
then matches exactly the so-called \emph{first-order stochastic dominance} at the level of measures.
In algebraic terms, for $p, q \in \R_+[X_1, \ldots, X_d]$ we have $p \le q$ if and only if one can increase the exponents in some of the terms of $p$ so as to obtain $q$.
Then in order for $p, q \in \R_+[X_1, \ldots, X_d]$ to be comparable at all, it is necessary for the sum of the coefficients to be the same, i.e.~we must have $p(1,\ldots,1) = q(1,\ldots,1)$. 
In particular, we indeed have $1 \not\ge 0$ and $1 \not\le 0$ in this preordered semiring.

Nevertheless, it is of interest to know when a catalytic preordering certificate \eqref{cat_cert} and/or an asymptotic preordering certificate \eqref{asymp_cert} exist.
Since $\R_+[X_1,\ldots,X_d]$ is isomorphic to the semiring of finitely supported measures on $\N^d$ with convolution as multiplication, these are equivalently questions about random walks on $\N^d$.
In particular, the asymptotic preordering detects when one random walk will dominate another one at late times, in the sense of its distribution being further ``upwards'' componentwise in $\N^d$.
Our Vergleichsstellens\"atze are exactly the right tool to detect when this dominance occurs by relating it to the spectral preorder, which can be calculated very concretely.
Of course, among the relevant spectral preorder relations are the point evaluations like
\[
	p(r_1, \ldots, r_d) > q(r_1, \ldots, r_d) \qquad \forall r_1, \ldots, r_d > 1.
\]
But something special happens at the point $(1,\ldots,1)$: due to $p(1,\ldots,1) = q(1,\ldots,1)$, which is the normalization of probability, it turns out that infinitesimal information around the evaluation homomorphism $f \mapsto f(1,\ldots,1)$ must be taken into account. This comes in the form of the inequalities
\begin{equation}
	\label{infinitesimal_intro}
	\frac{\partial p}{\partial X_i}(1,\ldots,1) < \frac{\partial q}{\partial X_i}(1,\ldots,1) \qquad \forall i = 1, \ldots, d.
\end{equation}
In terms of our formalism, these arise because the map
\[
	\frac{\partial}{\partial X_i}|_{(1,\ldots,1)} \: : \: \R_+[X_1,\ldots,X_d] \longrightarrow \R
\]
is a monotone derivation with respect to the evaluation homomorphism at $(1,\ldots,1)$.
For more detail, we refer to \Cref{free_ex,free_ex2} and the companion paper~\cite{arw}, where this is done more generally for random walks with compactly supported steps on topological abelian groups in general.

As this example may already indicate,
the technical challenges that arise in dealing with preordered semirings with $1 \not\ge 0$ are substantially greater than in the earlier case with $1 \ge 0$ that was considered in Part I.
A cleaner example illustrating some of the difficulty is the semifield $F \coloneqq \R_{(+)}[X] / (X^2)$ mentioned above, when equipped with the semiring preorder defined as
\[
	r_1 + s_1 X \le r_2 + s_2 X \quad : \Longleftrightarrow \quad r_1 = r_2 \:\:\land\:\: s_1 \le s_2.	
\]
Note that this is exactly the semiring preorder generated by $X \ge 1$.
Now the only homomorphism $F \to \R_+$ is the projection $r + s X \mapsto r$, and the only homomorphism $F \to \TR_+$ to the tropical reals $\TR_+$ is the degenerate one mapping every nonzero element to $1$.
Therefore if we were to use the same definition of the spectral preordering as in Part I, where this involved only monotone homomorphisms with values in $\R_+$ and $\TR_+$, then this spectral preordering would degenerate completely and could not display any interesting relation to the algebraic preordering on $F$.

As indicated already by~\eqref{infinitesimal_intro}, our solution to this problem is to enlarge the spectrum by infinitesimal information in the form of monotone derivations.
In fact, we will consider the preordered semifield $\R_{(+)}[X] / (X^2)$ semifield itself as another test object, so that the structure of other preordered semirings can also be probed through monotone homomorphisms with values in $\R_{(+)}[X] / (X^2)$. The two components of such a homomorphism form a pair $(\phi, D)$, where $\phi : S \to \R_+$ is a homomorphism that is \newterm{degenerate} in the sense that
\[
	x \le y \quad \Longrightarrow \quad \phi(x) = \phi(y),
\]
and a monotone \newterm{$\phi$-derivation} $D : S \to \R$, which is an additive monotone map satisfying the Leibniz rule with respect to $\phi$,
\[
	D(xy) = \phi(x) D(y) + D(x) \phi(y).
\]
Furthermore, we will also have to consider the \newterm{opposite semifields} $\R_+^\op$ and $\TR_+^\op$ as test objects.
This is not so surprising in light of the fact that reversing the preorder on a preordered semiring $S$ results in another preordered semiring $S^\op$.

In full technical detail, and using notions that will be introduced in the main text, our Vergleichsstellensatz for catalytic certificates is then the following.
We will prove it as \Cref{main1}.

\begin{thm}
	\label{intro_main1}
	Let $S$ be a zerosumfree preordered semidomain with a power universal pair $u_-, u_+ \in S$ and such that:
	\begin{itemize}
		\item $S / \!\sim$ has quasi-complements and quasi-inverses.
		\item $\Frac(S / \!\sim) \otimes \Z$ is a finite product of fields.
	\end{itemize}
	Let nonzero $x, y \in S$ with $x \sim y$ satisfy the following:
	\begin{itemize}
		\item For every nondegenerate monotone homomorphism $\phi : S \to \mathbb{K}$ with trivial kernel and $\mathbb{K} \in \{\R_+, \R_+^\op, \TR_+, \TR_+^\op\}$,
			\begin{equation}
				\label{phi_intro}
				\phi(x) < \phi(y).
			\end{equation}
		\item For every monotone additive map $D : S \to \R$, which is a $\phi$-derivation for some degenerate homomorphism $\phi : S \to \R_+$ with trivial kernel and satisfies $D(u_+) = D(u_-) + 1$,
			\begin{equation}
				\label{D_intro}
				D(x) < D(y).
			\end{equation}
	\end{itemize}
	Then there is nonzero $a \in S$ such that $a x \le a y$. 

	Moreover, if $S$ is also a semialgebra, then it is enough to consider $\R_+$-linear derivations $D$ in the assumptions.
\end{thm}

Conversely, if $a x \le a y$ holds for some nonzero $a$, then the same spectral conditions~(\ref{phi_intro}, \ref{D_intro}) are trivially implied with non-strict inequalities $\le$ in place of $<$.
The first two itemized assumptions on $S$ can be thought of as saying that although we do not require $1 \ge 0$ to be the case, this should nevertheless not fail \emph{too} badly, in the sense that the preorder relation $\le$ must still be suitably large.
In the final part of the theorem statement, the term \newterm{semialgebra} refers to the case in which $S$ comes equipped with a scalar multiplication by $\R_+$.

We do not currently have a Vergleichsstellensatz for asymptotic certificates that would apply at the same level of generality, but we do have one that applies under somewhat stronger assumptions on $S$ (which still fall short of assuming $1 \ge 0$).
This takes the following form, which in its conclusions precisely matches our main result of Part I.
It will be proven as \cref{main2}.

\begin{thm}
	\label{main2_intro}
	Let $S$ be a preordered semiring with a power universal element $u \in S$.
	Suppose that for some $d \in \N$, there is a surjective homomorphism $\|\cdot\| : S \to \R_{>0}^d \cup \{0\}$ with trivial kernel and such that
	\[
		a \le b \quad \Longrightarrow \quad \|a\| = \|b\| \quad \Longrightarrow \quad a \sim b.
	\]
	Let $x,y \in S$ be nonzero with $\|x\| = \|y\|$.
	Then the following are equivalent:
	\begin{enumerate}
		\item\label{spectral_ineqs_intro}
			\begin{itemize}
				\item For every nondegenerate monotone homomorphism $\phi : S \to \mathbb{K}$ with trivial kernel and $\mathbb{K} \in \{\R_+, \R_+^\op, \TR_+, \TR_+^\op\}$,
					\[
						\phi(x) \le \phi(y).
					\]
				\item For every $i = 1,\ldots,d$ and monotone $\|\cdot\|_i$-derivation $D : S \to \R$ with $D(u) = 1$,
					\[
						D(x) \le D(y).
					\]
			\end{itemize}
		\item\label{alg_ineqs_intro} For every $\eps > 0$, we have
			\[
				x^n \le u^{\lfloor \eps n \rfloor} y^n	\qquad \forall n \gg 1.
			\]
	\end{enumerate}
	Moreover, suppose that the inequalities in~\ref{spectral_ineqs_intro} are all strict.
	Then also the following hold:
	\begin{enumerate}[resume]
		\item There is $k \in \N$ such that
			\[
				u^k x^n \le u^k y^n \qquad \forall n \gg 1.
			\]
		\item If $y$ is power universal as well, then
			\[
				x^n \le y^n \qquad \forall n \gg 1.
			\]
		\item There is nonzero $a \in S$ such that
			\[
				a x \le a y.
			\]
			Moreover, there is $k \in \N$ such that $a \coloneqq u^k \sum_{j=0}^n x^j y^{n-j}$ for any $n \gg 1$ does the job.
	\end{enumerate}
	Finally, if $S$ is also a semialgebra, then all statements also hold with only $\R_+$-linear derivations $D$ in \ref{spectral_ineqs_intro}.
\end{thm}

We again refer to~\cite{arw} for the application to random walks, which gives much stronger results than what has been achieved with purely probabilistic methods so far, and to~\cite{major} for another application to information theory.

\subsection*{Overview}

We now give some indication of the content of each section.

\begin{itemize}
	\item \Cref{background} summarizes the main definitions and results of Part I, so that the present paper can be read independently of~\cite{partI}.
	\item \Cref{further} introduces a few additional relevant definitions and makes some basic observations that will be used in the remainder of the paper.
\end{itemize}

The next few sections are devoted to developing some structure theory of preordered semifields.
This builds the technical groundwork for our main results, but we also expect it to be relevant for future work in the area.

\begin{itemize}
	\item \Cref{intermezzo} proves a number of important and surprisingly strong inequalities in preordered semirings and semifields by elementary means.
		Although we will not dwell on this relation further, readers with a good background in probability theory will be able to interpret many of those inequalities in terms of second-order stochastic dominance.
	\item \Cref{malt_full} introduces multiplicatively Archimedean fully preordered semifields and analyzes their structure, resulting in a classification into five types. 
		\Cref{layer_preorder} provides a sense in which these preordered semifields are the building blocks of all totally preordered semifields.
		And since the latter are a stepping stone for analyzing the structure of \emph{any} preordered semifield by our first Vergleichsstellensatz from Part I (\Cref{semifield_main}), the results of \Cref{malt_full} are an important part of the structure theory of preordered semifields in general.
	\item \Cref{ambient} introduces a certain derived preorder relation on any preordered semifield, the so-called \emph{ambient preorder}. 
		The main use of this construction is that even if $1$ and $0$ are not preordered relative to each other in a given preordered semifield, they often will be with regards to the ambient preorder.
		With some effort, this then allows us to leverage the results of Part I, making them apply in particular to multiplicatively Archimedean fully preordered semifields as studied in \Cref{malt_full}.
		This results in two embeddings theorems, namely \Cref{real_trop} and \Cref{arctic_deriv}, which constitute the main results of this section.
\end{itemize}
The final two sections contain our two main results, the proofs of which crucially rely on the auxiliary results of \Cref{intermezzo,malt_full,ambient}.
\begin{itemize}
	\item \Cref{n2} combines the results obtained so far and derives a catalytic Vergleichsstellensatz, namely the \Cref{intro_main1} quoted above.
		While the main applications appear elsewhere, \Cref{free_ex} showcases the application to polynomials mentioned above.
	\item \Cref{n3} then proves an asymptotic Vergleichsstellensatz, namely the above \Cref{main2_intro}. The proof is based on \Cref{intro_main1} and the compactness of the \emph{test spectrum} (\Cref{chaus2}), as introduced in \Cref{test_spectrum}. \Cref{free_ex2} briefly illustrates the statement in the same polynomial semiring example as before.
\end{itemize}

\subsection*{Relation to existing literature}

The lack of additive inverses in semirings frequently forces us to use very different methods than those used standardly in the proofs of Positivstellens\"atze. 
This relates to the additive structure often being considered primary in real algebra, which manifests itself e.g.\ in the definition of Archimedeanicity or in the use of functional-analytic methods like decompositions of states into pure states.
In contrast, our methods put the emphasis on the multiplicative structure and exploit the existence of multiplicative inverses in semifields.
A good case in point are the proofs of the various inequalities derived in \Cref{intermezzo}.

Nevertheless, some of our methods overlap with those used in proofs of Positivstellens\"atze.
This applies in particular to the following ideas:
\begin{itemize}
	\item The reduction to total semifield preorders of \cref{semifield_main} mirrors a very general proof technique for separation theorems, used already e.g.~in Artin's solution to Hilbert's 17th problem~\cite[Lemma~1.4.4]{marshall}.
	\item Our use of compactness arguments for the spectrum to achieve strict is parallel to standard arguments in real algebra. Compare e.g.~\cite[Theorem~5.7.2]{marshall} with our \cref{n3}.
	\item Our use of derivations to capture infinitesimal information is also reminiscent of the use of derivative information, as for example in Positivstellensatz-type results of Marshall~\cite[Theorem~2.3]{Marshall2006}, Scheiderer~\cite[Corollary~3.6]{Scheiderer2005} and Burgdorf, Scheiderer and Schweighofer~\cite[Theorem~7.8]{BSS}.
\end{itemize}

\section{Background from Part I}
\label{background}

Here, we recall the basic definitions around semirings together with the main new definitions and results developed in Part I~\cite{partI} as far as they are relevant to this paper.
In the following, all unreferenced definitions and results are standard. Those of Part I are referenced with their theorem number prefixed by ``I''.

A \newterm{commutative semiring} is a set $S$ together with two commutative monoid structures $(S,+,0)$ and $(S,\cdot,1)$ such that the multiplication $\cdot$ distributes over the addition $+$ and $1 \cdot 0 = 0$. Since we will not consider the noncommutative case at all, we simply use the term \newterm{semiring} as short for ``commutative semiring''. The set of multiplicatively invertible elements in a semiring $S$ is denoted $S^\times$. There is a unique semiring homomorphism $\N \to S$, given by $n \mapsto n1$, and we often abuse notation by writing $n$ instead of $n1$. A semiring $S$ is \newterm{zerosumfree} if $x + y = 0$ in $S$ implies $x = y = 0$. It has \newterm{quasi-complements} if for every $x \in S$ there are $y \in S$ and $n \in \N$ such that $x + y = n$ (I.2.19).

A \newterm{semifield} $F$ is a semiring such that $F^\times = F \setminus\{0\}$. A semifield is \newterm{strict} if $1$ has no additive inverse, or equivalently if $F^\times$ is closed under addition, or yet equivalently if $F$ is zerosumfree. Every semifield that is not a field is strict. A \newterm{semidomain} is a semiring without zero divisors and such that $1 \neq 0$. A semidomain $S$ has a \newterm{semifield of fractions} $\Frac(S)$ together with a homomorphism $S \to \Frac(S)$ that is the initial semiring homomorphism from $S$ into a semifield. $\Frac(S)$ is a strict semifield if and only if $S$ is zerosumfree.

A \newterm{preorder} is a binary relation that is reflexive and transitive. We typically denote a preorder by $\le$, where the symbols $\ge$, $<$ and $>$ have their standard induced meaning. We write $\preceq$ when another preorder symbol is needed. We also use the following two derived relations:
\begin{itemize}
	\item $\approx$ is the largest equivalence relation contained in $\le$. In other words, $x \approx y$ is shorthand for $(x \le y) \:\land\: (y \le x)$.
	\item $\sim$ is the smallest equivalence relation containing $\le$. In other words, $x \sim y$ is shorthand for the existence of a finite sequence $x = z_0, z_1, \ldots, z_n = y$ such that $z_i \le z_{i+1}$ or $z_i \ge z_{i+1}$ for all $i$.
\end{itemize}
A preorder is \newterm{total} if $x \le y$ or $y \le x$ for all $x$ and $y$. A map $f$ between preordered sets is \newterm{monotone} if $x \le y$ implies $f(x) \le f(y)$.
A monotone map $f$ is an \newterm{order embedding} if the converse holds as well. If $X$ is a preordered set, then $X^\op$ denotes the preordered set with the \newterm{opposite preorder} in which $x \le y$ holds if and only if $y \le x$ in $X$.

A \newterm{preordered semiring} is a semiring together with a preorder relation such that for every $a \in S$, both addition by $a$ and multiplication by $a$ are monotone maps $S \to S$ (I.3.7). Note that $1 \ge 0$ is not required. If $S$ is a preordered semiring, then so is $S^\op$. \newterm{Preordered semifields} and \newterm{preordered semidomains} (I.3.16/21) are preordered semirings that are semifields, respectively semidomains, with the additional condition that $x \approx 0$ implies $x = 0$; for a preordered semifield, this means equivalently that $1 \not\le 0$ or $1 \not\ge 0$.
If $S$ is a preordered semidomain, then $\Frac(S)$ is a preordered semifield in which
\begin{equation}
	\label{frac_preorder}
	\frac{x}{a} \le \frac{y}{b} \quad :\Longleftrightarrow \quad \exists r \in S \setminus \{0\} : \; x b \le y a.
\end{equation}
With this definition, $\Frac(S)$ is a preordered semifield (I.3.22) such that the canonical homomorphism $S \to \Frac(S)$ is monotone (I.3.23).

The following is a central result in the theory of preordered semifields.

\begin{thm}[I.6.6]
	\label{semifield_main}
	Let $F$ be a preordered semifield. Then the preorder on $F$ is the intersection of all its total semifield preorder extensions.	
\end{thm}

In other words, if $x \not\le y$ in $F$, then the preorder on $F$ can be extended to a total semifield preorder $\preceq$ such that still $x \not\preceq y$, which by totality in particular implies $x \succ y$.

A preordered semifield $F$ is \newterm{multiplicatively Archimedean} if $x^k \le y$ for all $k \in \N$ implies $x \le 1$ for all $x, y \in F^\times$ (Definition~I.4.1). The paradigmatic examples of multiplicatively Archimedean semifields are $\R_+$ with its usual algebraic structure as well as the \newterm{tropical reals} $\TR_+$, defined as
\[
	\TR_+ \,\coloneqq\, (\R_+, \max, \cdot) \,\cong\, (\R \cup \{-\infty\}, \max, +),
\]
where the first equation is the \newterm{multiplicative picture} of $\TR_+$ and the second the \newterm{additive picture}. The isomorphism between them is given by the logarithm (with any base). As far as it matters, we always specify explicitly which picture of $\TR_+$ we use, with the multiplicative picture often being preferred. The tropical reals also contain the \newterm{Boolean semifield} $\B = \{0,1\}$, in which $1 + 1 = 1$. It is the terminal object in the category of strict semifields and semiring homomorphisms.
All the strict semifields mentioned in this paragraph are preordered semifields with respect to the standard preorder relation.

\begin{thm}[I.4.2]
	\label{real_or_tropical}
	Let $F$ be a multiplicatively Archimedean totally preordered semifield. Then $F$ order embeds into one of the following:
	\[
		\R_+, \quad \R_+^\op, \quad \TR_+, \quad \TR_+^\op.
	\]
\end{thm}

A surprising construction on preordered semifields is the \newterm{categorical product} (I.3.19). If $F_1$ and $F_2$ are preordered semifields, then their categorical product has underlying set
\[
	(F_1^\times \times F_2^\times) \cup \{(0, 0)\}
\]
and carries the componentwise algebraic operations and the componentwise preorder, and it is straightforward to see that it is a preordered semifield again (having the universal property of a categorical product).

Next, we recall the relevant definitions around polynomial growth (I.3.27). If $S$ is a preordered semiring, then a pair of nonzero elements $u_-, u_+ \in S$ is a \newterm{power universal pair} if $u_- \le u_+$ and for every nonzero $x,y \in S$ with $x \le y$, there is $k \in \N$ such that
\beq
	\label{pgrowth}
	y u_-^k \le x u_+^k.
\eeq
It follows that the same property holds already if merely $x \sim y$ (I.3.28).
We say that $S$ is of \newterm{polynomial growth} if it has a power universal pair. A \newterm{power universal element} is $u \in S$ such that $(1,u)$ is a power universal pair. A preordered semifield of polynomial growth has a power universal element given by $u \coloneqq u_+ u_-^{-1}$. Therefore when working with preordered semifields of polynomial growth, we will always work with a power universal element. On the other hand, there are preordered semirings of polynomial growth that do not have a power universal element but merely a power universal pair, such as the polynomial semiring $\N[X_1,\ldots,X_d]$ equipped with the coefficientwise preorder (I.3.37).

\Cref{real_or_tropical} implies the following:

\begin{cor}[I.4.3]
	\label{arch_trunc}
	Let $F$ be a totally preordered semifield of polynomial growth. Then there is a monotone homomorphism $\phi : F \to \mathbb{K}$ with $\mathbb{K} \in \{\R_+, \R_+^\op, \TR_+, \TR_+^\op\}$ such that $\phi(u) > 1$ for every power universal element $u > 1$.
\end{cor}

Based on all of these results, we had also developed an abstract Vergleichsstellensatz for preordered semirings $S$ with $1 > 0$ and having a power universal element $u$. This uses the \newterm{test spectrum} (I.7.1), which under these assumptions\footnote{In \Cref{n2}, we will introduce a different definition of test spectrum applying to a different class of preordered semirings. A suitable general definition of spectrum of a preordered semiring remains to be found.} is the disjoint union of the monotone homomorphisms to $\R_+$ and $\TR_+$, where the latter are suitably normalized,
\begin{align*}
	\Sper{S} \coloneqq	& \{ \textrm{ monotone hom } \phi : S \to \R_+ \} \\
				& \sqcup \{ \textrm{ monotone hom } \phi : S \to \TR_+ \textrm{ with } \phi(u) = e \},
\end{align*}
and equipped with a certain topology that turns this set into a compact Hausdorff space (I.7.9).
The resulting Vergleichsstellensatz reads as follows.

\begin{thm}[I.7.15]
	\label{Imain}
	Let $S$ be a preordered semiring with $1 > 0$ and a power universal element $u$. Then for nonzero $x, y \in S$, the following are equivalent:
	\begin{enumerate}
		\item $\phi(x) \le \phi(y)$ for all $\phi \in \Sper{S}$.
		\item For every $\eps > 0$ we have
			\[
				x^n \le u^{\lfloor \eps n \rfloor} y^n \qquad \forall n \gg 1.
			\]
	\end{enumerate}
	Moreover, if $\phi(x) < \phi(y)$ for all $\phi \in \Sper{S}$, then also the following hold:
	\begin{enumerate}[resume]
		\item\label{Iasymp} There is $k \in \N$ such that
			\[
				u^k x^n \le u^k y^n \qquad \forall n \gg 1.
			\]
		\item If $y$ is a power universal element itself, then also
			\[
				x^n \le y^n \qquad \forall n \gg 1.
			\]
		\item\label{Icat} There is nonzero $a \in S$ such that
			\[
				a x \le a y.
			\]
			More concretely, $a = u^k \sum_{j=0}^n x^j y^{n-j}$ works for some $k$ and all $n \gg 1$.
	\end{enumerate}
\end{thm}

This theorem specializes to a version of Strassen's Vergleichsstellensatz~\cite[Corollary~2.6]{strassen} for $u = 2$.
In Part I, we also showed how the classical Positivstellensatz of Krivine--Kadison--Dubois can be derived from the latter, and therefore also from our Vergleichsstellensatz, in an elementary way (I.8.4).

\section{Further basic definitions and observations}
\label{further}

Later in the paper, and in particular in the statements of our main results, we will use a few additional concepts that we introduce now.

\begin{defn}
	A semidomain $S$ has \newterm{quasi-inverses} if for every nonzero $x \in S$ there are $n \in \Nplus$ and $y \in S$ with $xy = n$.
\end{defn}

As for some basic examples, every semifield trivially has quasi-inverses. $\N$ also has quasi-inverses. So does every semidomain of the form $\Nplus^n \cup \{(0,\ldots,0)\}$, where the semiring structure is given by the componentwise operations.

The relevance of quasi-inverses is explained by the following observation.

\begin{lem}
	\label{frac_quasi}
	Let $S$ be a semidomain with quasi-complements and quasi-inverses. Then the semifield of fractions $\Frac(S)$ has quasi-complements too.
\end{lem}

\begin{proof}
	We first construct quasi-complements for fractions of the form $\frac{1}{b} \in \Frac(S)$ for nonzero $b \in S$.
	We choose a quasi-inverse $c \in S$ with $bc = n \in \Nplus$. Then $\frac{1}{b} = \frac{c}{n}$ in $\Frac(S)$. Choosing a quasi-complement $d$ for $c$ makes $\frac{c}{n} + \frac{d}{n}$ into a rational number. By further adding a suitable additional rational number, we therefore obtain a natural number, as was to be shown.

	For arbitrary $\frac{a}{b} \in \Frac(S)$, we choose a quasi-complement $c \in S$ for $a$, resulting in $a + c = m \in \N$. Then
	\[
		\frac{a}{b} + \frac{c}{b} = \frac{m}{b},
	\]
	and the claim follows by the previous case upon multiplying any quasi-complement of $\frac{1}{b}$ by $m$.
\end{proof}

\begin{lem}
	\label{frac_sim}
	Let $S$ be a preordered semidomain. Then the identity map induces an isomorphism of semifields
	\[
		\Frac(S / \!\sim) \cong \Frac(S) / \!\sim.
	\]
\end{lem}

Note that this slightly abuses notation by using the same symbol $\sim$ for the equivalence relation generated by the preorder on $S$ and for the one generated by the preorder on $\Frac(S)$.

\begin{proof}
	In both semifields, the elements are given by equivalence classes of formal fractions, and we write $\equiv$ for the equivalence relation in each case.
	In $\Frac(S / \!\sim)$, the equivalence relation can be described as the one generated by the following two basic equivalences:
	\begin{equation}
		\label{frac_sim1}
		\frac{x}{a} \equiv \frac{rx}{ra} \:\textrm{ if }\: r \in S \setminus \{0\},
		\qquad 
		\frac{x}{a} \equiv \frac{y}{a} \:\textrm{ if }\: x \le y,
		\qquad
		\frac{x}{a} \equiv \frac{x}{b} \:\textrm{ if }\: a \le b.
	\end{equation}
	On $\Frac(S) / \!\sim$, we still have the first kind of generating equivalences, but the second and third kind is replaced by
	\begin{equation}
		\label{frac_sim2}
		\frac{x}{a} \equiv \frac{y}{b} \:\textrm{ if }\: \exists r \in S \setminus \{0\} : x b r \le y a r.
	\end{equation}
	If either of the equivalences in~\eqref{frac_sim1} holds, then it is clear that~\eqref{frac_sim2} holds as well.
	Conversely if~\eqref{frac_sim2} holds, then we get
	\[
		\frac{x}{a} \equiv \frac{x b r}{a b r} \equiv \frac{y a r}{a b r} \equiv \frac{y}{b}
	\]
	with respect to~\eqref{frac_sim1}.
	This shows that the two equivalence relations coincide, and therefore the identity map induces an isomorphism of semifields as claimed.
\end{proof}

\begin{defn}
	Let $S$ and $T$ be preordered semirings. A monotone homomorphism $\phi : S \to T$ is \newterm{degenerate} if it factors through $S / \!\sim$, or equivalently if for all $x, y \in S$,
	\[
		x \le y \quad\Longrightarrow\quad \phi(x) = \phi(y).
	\]
	Otherwise $\phi$ is \newterm{nondegenerate}.
\end{defn}

The next definition refers to unordered structures only, but will later be used in the context of preordered semirings $S$ with monotone $D$ and degenerate $\phi$.

\begin{defn}
	Let $S$ be a semiring and $\phi : S \to \R_+$ a homomorphism. Then a \newterm{$\phi$-derivation} is a map $D : S \to \R$ such that the \newterm{Leibniz rule}
	\[
		D(xy) = \phi(x) D(y) + D(x) \phi(y)
	\]
	holds for all $x, y \in S$.
\end{defn}

Note that the set of $\phi$-derivations is a vector space over $\R$.
And if $S$ is a preordered semiring, then the set of monotone $\phi$-derivations is a convex cone inside this vector space.
In either case, the $\phi$-derivations can be thought of geometrically as tangent vectors to the spectral point $\phi$. 

Our results will take a slightly stronger form for semirings which have a scalar multiplication by $\R_+$.

\begin{defn}
	A \newterm{semialgebra}\footnote{Since we will not consider scalar multiplication by anything other than $\R_+$, we omit explicit mention of the semifield of scalars.} $S$ is a semiring together with a scalar multiplication
	\begin{align*}
		\R_+ \times S	& \longrightarrow S	\\
		(r, x)		& \longmapsto rx
	\end{align*}
	that is a commutative monoid homomorphism in each argument and satisfies the following additional laws:
	\begin{itemize}
		\item $1x = x$,
		\item $(rx)(sy) = (rs)(xy)$.
	\end{itemize}
\end{defn}

In other words, a semialgebra is a semiring which at the same time is a semimodule over $\R_+$~\cite[Chapter~14]{golan} in such a way that the multiplication of $S$ is bilinear.\footnote{The definition of semimodule includes the additional law $r(sx) = (rs)x$, but this is implied by the ones we have assumed.}
As with algebras over commutative rings, a semialgebra is equivalently a semiring $S$ equipped with a semiring homomorphism $\R_+ \to S$.

\begin{defn}
	For a semiring $S$, we write $S \otimes \Z$ for the \newterm{ring generated by $S$}, i.e.~the initial object in the category of commutative rings equipped with a semiring homomorphism $S \to S \otimes \Z$.
\end{defn}

As is well-known, one obtains $S \otimes \Z$ by applying the Grothendieck construction to $S$, which means taking the elements of $S \otimes \Z$ to be formal differences of elements of $S$.
If $S$ is a semialgebra, then $S \otimes \Z$ is an $\R$-algebra in the obvious way.

Changing topic, we will also need a piece of terminology for preorders that are not necessarily total, but merely total on connected components.

\begin{defn}
	\label{totalfull}
	\begin{enumerate}
		\item A preorder relation $\le$ is \newterm{full} if
			\[
				x \sim y \quad \Longrightarrow \quad x \le y \enspace \lor \enspace x \ge y
			\]
			for all $x$ and $y$.
	\end{enumerate}
\end{defn}

Here is an equivalent characterization.

\begin{lem}
	\label{fullchar}
	A preorder is full if and only if the following holds for all $a,x,y$:
	\[
		a \le x,y \enspace \lor \enspace x,y \le a \quad \Longrightarrow \quad x \le y \enspace \lor \enspace x \ge y.
	\]
\end{lem}

\begin{proof}
	This condition is clearly necessary. For sufficiency, let us assume that the condition holds. We temporarily write $x \simeq y$ as shorthand for $x \le y \: \lor \: x \ge y$. In order to prove that this is indeed the equivalence relation generated by $\le$, which is $\sim$, we need to show that it is transitive. So let $x \simeq y \simeq z$. If $x \le y \le z$, then we can conclude $x \le z$ and hence $x \simeq z$ by transitivity of $\le$, and likewise if $x \ge y \ge z$. If $x \le y \ge z$ instead, then applying the assumption with $a \coloneqq y$ results in the desired $x \simeq z$, and similarly if $x \ge y \le z$.
\end{proof}

Clearly every total preorder is full, but not conversely. For example, the trivial preorder on any set is full.

\begin{lem}
	\label{full_q}
	Let $F$ be a preordered semifield which is isomorphic to $\Q_+$ as a semifield.
	Then the preorder on $\Q_{>0}$ is one of the following:
	\begin{enumerate}
		\item $r \le s$ for all $r, s \in \Q_{>0}$.
		\item The usual order or its opposite.
		\item $r \le s$ if and only if $r = s$ for all $r, s \in \Q_{>0}$.
	\end{enumerate}
\end{lem}

\begin{proof}

\end{proof}

\section{Inequalities in preordered semirings and semifields}
\label{intermezzo}

In this section, we prove some elementary but nontrivial results on implications between inequalities in preordered semirings and semifields (with the focus on the latter).
These will form an important building block for the deeper results that we develop in the subsequent sections.

\subsection*{Chaining inequalities in preordered semirings}

We start with some observations on chaining inequalities. In this subsection, everything takes place in a preordered semiring $S$, without any additional hypotheses.

\begin{lem}
	\label{sharpen_add}
	If $a + x \le a + y$, then also $a + nx \le a + ny$ for every $n \in \N$.
\end{lem}

\begin{proof}
	The claim is trivial for $n = 0$. For the induction step, we use
	\begin{align*}
		a + (n+1)x	& = (a + x) + nx \le (a + y) + nx \\
				& = y + (a + nx) \le y + (a + ny) = a + (n+1)y. \qedhere
	\end{align*}
\end{proof}

We will routinely use this trick in the rest of the paper and simply call it \newterm{chaining}. A stronger statement along the same lines is as follows.

\begin{lem}
	\label{add_to_mult}
	Let $p = \sum_i r_i X^i \in \N[X]$ be any polynomial with coefficients $r_i > 0$ for all $i = 0, \ldots, \deg(p)$. If $x + 1 \le y + 1$ in $S$, then also $p(x) \le p(y)$, 
\end{lem}

\begin{proof}
	We first prove that
	\[
		x p(y) + 1 \le y p(y) + 1
	\]
	for any such polynomial $p$. Using well-founded induction, it is enough to show that if this holds for $p$, then it also holds for $p + 1$ and for $Xp + 1$. Indeed for the former,
	\[
		x ( p(y) + 1) + 1 \le x p(y) + y + 1 \le y ( p(y) + 1 ) + 1,
	\]
	where we first use the overall assumption and then the induction assumption, whereas for the latter similarly
	\begin{align*}
		x (y p(y) + 1 ) + 1	& = x y p(y) + x + 1 \le x y p(y) + y + 1 \\[2pt]
					& = y ( x p(y) + 1 ) + 1 \le y ( y p(y) + 1 ) + 1,
	\end{align*}
	as was to be shown.

	Getting to the claim itself, we use the same type of induction on $p$. Now the first case is trivial, while the second case has induction assumption $p(x) \le p(y)$ and proves that
	\[
		1 + x p(x) \le 1 + x p(y) \le 1 + y p(y),
	\]
	where the first step is by induction assumption and the second by the auxiliary statement above.
\end{proof}

\subsection*{Some inequalities in preordered semifields}

In this subsection and the following ones, everything takes place in a preordered semifield $F$.
The next few results will be a working horse for us in~\Cref{malt_full}.

\begin{lem}
	\label{power_lemma}
	Let $x \in F^\times$. If $x + x^{-1} \ge 2$, then also the following hold for all $m,n \ge 1$:
	\begin{enumerate}[label=(\roman*)]
		\item\label{2nd_order_dominance} $m x^n + n x^{-m} \ge m + n$.
		\item\label{exponential_better} $2^{n-1}(x^n + x^{-n}) \ge (x + x^{-1})^n$.
		\item\label{supermodular} $x^{m+n} + 1 \ge x^m + x^n$.
	\end{enumerate}
	If $x + x^{-2} > 2$, then these inequalities also hold strictly.
\end{lem}

\begin{proof}
	We focus on the non-strict inequality case, since the strict one follows the same way upon noting that at least one inequality in each chain of inequalities will be strict.
	We first prove a few auxiliary statements that are special cases of the above claims. We routinely use the assumption in the form $x^{j+1} + x^{j-1} \ge 2 x^j$.
	\begin{enumerate}
		\item\label{simple_2norder} $x^2 + x^{-2} \ge x + x^{-1}$.

			Indeed, repeatedly applying the assumption gives
			\begin{align*}
				2 (x + x^{-1}) (x^2 + x^{-2})	& = 2 x^3 + 2 x + 2 x^{-1} + 2 x^{-3} \\
								& \ge 2 x^3 + x + 2 + x^{-1} + 2 x^{-3} \\
								& \ge x^3 + 2 x^2 + 2 + 2 x^{-2} + x^{-3} \\
								& \ge x^3 + x^2 + 2 x + 2 x^{-1} + x^{-2} + x^{-3} \\
								& \ge x^3 + x^2 + x + 2 + x^{-1} + x^{-2} + x^{-3} \\
								& \ge 3 x^2 + 2 + 3 x^{-2} \\
								& \ge 2 x^2 + 2 x + 2 x^{-1} + 2 x^{-2} \\
								& \ge 2(x^2 + 2 + x^{-2}) \\
								& = 2(x + x^{-1}) (x + x^{-1}),
			\end{align*}
			so that the claim follows upon cancelling $2(x + x^{-1})$.
		\item\label{basic_2norder} The map $n \mapsto x^n + x^{-n}$ is monotone in $n \in \N$.

			Indeed using induction on $n$, the inequality
			\[
				x^{n+1} + x^{-(n+1)} \ge x^n + x^{-n}	
			\]
			holds by assumption in the base case $n = 0$. For the induction step from $n$ to $n+1$, we compute
			\begin{align*}
				(x + x^{-1}) (x^{n+2} + x^{-(n+2)})	& = x^{n+3} + x^{n+1} + x^{-(n+1)} + x^{-(n+3)} \\
									& \ge x^{n+3} + x^{n-1} + x^{-(n-1)} + x^{-(n+3)} \\
									& \ge x^{n+2} + x^n + x^{-n} + x^{-(n+2)} \\
									& = (x + x^{-1}) (x^{n+1} + x^{-(n+1)}),
			\end{align*}
			and again cancel the term $x + x^{-1}$ from both sides. Here, the first inequality holds by induction assumption and the second by the previous item.
		\item\label{basic_2norder2} In particular, we therefore have $x^n + x^{-n} \ge 2$ for all $n \in \N$.
		\item $x^{n+1} + n \ge (n+1) x$ for all $n \in \N$.
			
			This claim is trivial for $n = 0$ and holds by assumption for $n = 1$. For all other $n$ we use induction, distinguishing the case of even exponent,
			\begin{align*}
				x^{2m+2} + (2m + 1)	& = x^{m+1} (x^{m+1} + x^{-(m+1)}) + 2m \\
							& \ge 2 x^{m + 1} + 2 m \\
							& \ge (2 m + 2) x,
			\end{align*}
			where the first inequality is by the previous item and the second by the induction assumption for $n = m$.
			The case of odd exponent is slightly more difficult: with $m \ge 1$,
			\begin{align*}
				(x + 1)(x^{2m+1} + 2m)	& = x^{2m + 2} + x^{2m + 1} + 2mx + 2m \\
							& \ge 2 x^{m+1} + x^{2m + 1} + 2mx + (2m - 1) \\
							& \ge 2 x^{m+1} + 2 m x^2 + x + (2m - 1) \\
							& \ge x^{m+1} + 2 m x^2 + (m + 2) x + (m - 1) \\
							& \ge 3 m x^2 + 3 x + (m - 1) \\
							& \ge (2 m + 1) x^2 + (2 m + 1) x \\
							& = (x + 1) (2 m + 1) x,
			\end{align*}
			where the first inequality uses $x^{m+1} + x^{-(m+1)} \ge 2$ as an instance of \ref{basic_2norder2}, the subsequent three use the induction assumption, and the final one is just the assumed $x + x^{-1} \ge 2$.
	\end{enumerate}
	We now prove the actual three claims.
	\begin{enumerate}[label=(\roman*)]
		\item This holds trivially for $n = 0$ or $m = 0$. For the induction step in $n$ assuming fixed $m \ge 1$, we use
			\begin{align*}
				(m + n) \left(m x^{n+1} + (n + 1) x^{-m} \right) & \ge m \left( m + n + 1 \right) x^n + n \left( m + n + 1 \right) x^{-m} \\[2pt]
					& \ge (m + n) (m + n + 1),
			\end{align*}
			where the first inequality holds because of $(m + n) x^{n + 1} + x^{-m} \ge (m + n + 1) x^n$ as a consequence of the previous item, and the second by the induction assumption.
		\item For any $n \ge 1$ and $j = 0,\ldots,n$, the inequality
			\[
				(n - j) x^n + j x^{-n} \ge n x^{n-2j}
			\]
			holds by the previous item. Together with standard identities for binomial coefficients, it gives
			\begin{align*}
				2^{n-1} \left( x^n + x^{-n} \right) & = \sum_{j=0}^{n-1} \binom{n-1}{n-j-1} x^n + \sum_{j=1}^n \binom{n-1}{j-1} x^{-n} \\[2pt]
								& = n^{-1} \sum_{j=0}^n \binom{n}{j} \left( (n-j) x^n + j x^{-n} \right) \\[2pt]
								& \ge \sum_{j=0}^n \binom{n}{j} x^{n-2j} = (x + x^{-1})^n, \\[2pt]
			\end{align*}
			as was to be shown.
		\item We derive this from \Cref{2nd_order_dominance},
			\begin{align*}
				(m + n) (x^{m+n} + 1)	& = (m x^n + n x^{-m}) x^m + (n x^m + m x^{-n}) x^n \\
							& \ge (m+n) (x^m + x^n),
			\end{align*}
			which already gives the claim upon dividing by $m + n$. \qedhere
	\end{enumerate}
\end{proof}

Inequalities of a similar flavour are now quite easy to derive.

\begin{lem}
	\label{power_skew3}
	Let $x \in F^\times$. If $x + x^{-1} \ge 2$, then also
	\[
		\binom{n+2}{2} x^n \le \binom{n+1}{2} x^{n+1} + \sum_{j=0}^n x^j
	\]
	for all $n \in \N$.
\end{lem}

\begin{proof}
	By \Cref{power_lemma}\ref{2nd_order_dominance}, we have
	\[
		(n-j+1) x^n \le (n-j) x^{n+1} + x^j,
	\]
	for all $j = 0,\ldots,n$, and this gives the claim upon summation over $j$.
\end{proof}

One of the auxiliary inequalities used in the proof of \Cref{power_lemma} is worth noting separately.

\begin{cor}
	\label{norder}
	Let $x \in F^\times$.
	\begin{enumerate}[label=(\roman*)]
		\item If $x + x^{-1} \ge 2$, then $x^n + x^{-n} \ge 2$ for all $n \in \N$.
		\item If $x + x^{-1} > 2$, then $x^n + x^{-n} > 2$ for all $n \in \Nplus$.
	\end{enumerate}
\end{cor}

Alternatively, this can also be regarded as a weakening of \Cref{power_lemma}\ref{exponential_better}.
The following two lemmas are again important technical results, the significance of which will become clearer in \Cref{malt_full}.

\begin{lem}
	\label{other_power_lemma}
	Let $x \ge 1$ in $F$ be such that 
	\[
		x^{n+1} + 1 \le x^n + 1
	\]
	for some $n \in \N$. Then also $(x + 1)^m \le 2^m x^n$ for all $m \in \N$.
\end{lem}

\begin{proof}
	We start with an auxiliary statement similar to \Cref{add_to_mult}: whenever $p = \sum_i r_i X^i \in \N[X]$ is any polynomial with coefficients $r_0,\ldots,r_{\deg(p)} \neq 0$, then 
	\beq
		\label{pminbound}
		p(x) \le \sum_i r_i x^{\min(i,n)}.
	\eeq
	To prove this, we use well-founded induction similar to the one in the proof of \Cref{add_to_mult}. The statement is trivial whenever $\deg(p) \le n$. When $\deg(p) > n$, we can write
	\[
		p = X^{\deg(p)} + X^{\deg(p) - (n + 1)} + \hat{p}
	\]
	for some $\hat{p} \in \N[X]$, where by the induction assumption the claim can be assumed to hold for the ``smaller'' polynomial $X^{\deg(p) - 1} + X^{\deg(p) - (n + 1)} + \hat{p}$, which differs from $p$ only in the exponent of the leading term. But then
	\begin{align*}
		p(x) = (x^{n + 1} + 1) x^{\deg(p) - (n + 1)} + \hat{p}(x) & \le (x^n + 1) x^{\deg(p) - (n + 1)} + \hat{p}(x) \\
									& = x^{\deg(p) - 1} + x^{\deg(p) - (n + 1)} + \hat{p}(x).
	\end{align*}
	Applying the induction assumption now proves the claim~\eqref{pminbound} upon using $\deg(p) > n$.

	Upon bounding the right-hand side of \eqref{pminbound} further using $x \ge 1$, we obtain the somewhat weaker bound
	\[
		p(x) \le \left( \sum_i r_i \right) x^n = p(1) x^n,
	\]
	which is more convenient since now the right-hand side is a mere monomial. The claim follows upon applying this statement to the polynomial $p \coloneqq (X + 1)^m$.
\end{proof}

\begin{lem}
	\label{nonarctic_bound}
	Let $x \ge 1$ in $F$ be such that $x^2 + 2 \ge 3 x$. Then also
	\[
		x^{n + 1} + 1 \ge 2 x^n
	\]
	for all $n \in \N$.
\end{lem}

\begin{proof}
	We use induction on $n$. The base case $n = 0$ is trivial by $x \ge 1$. The induction step from $n$ to $n+1$ is
	\begin{align*}
		(2 x + 1) (x^{n + 2} + 1)	& = 2 x^{n + 3} + x^{n + 2} + 2 x + 1 \\
						& \ge 2 x^{n + 3} + 2 x^{n + 1} + x + 1 \\
						& \ge x^{n + 3} + 3 x^{n + 2} + x + 1 \\
						& \ge x^{n + 3} + 2 x^{n + 2} + 2 x^{n + 1} + 1 \\
						& \ge x^{n + 3} + 2 x^{n + 2} + x^{n + 1} + 2 x^n \\
						& \ge x^{n + 3} + x^{n + 2} + 4 x^{n + 1} \\
						& \ge 4 x^{n + 2} + 2 x^{n + 1} \\
						& = (2 x + 1) 2 x^{n + 1},
	\end{align*}
	where each inequality step uses either the induction assumption or the assumed inequality $x^2 + 2 \ge 3 x$.
\end{proof}

Applying \Cref{power_lemma} to both $F$ and $F^\op$ gives a result which is worth stating separately, since it has some relevance to the \emph{arctic} case in \Cref{malt_full}.

\begin{lem}
	\label{power_square}
	Let $x \in F^\times$. If $x + x^{-1} \approx 2$, then also the following hold for all $m,n \in \N$:
	\begin{enumerate}[label=(\roman*)]
		\item $x^n + x^{-n} \approx 2$.
		\item $m x^n + n x^{-m} \approx m + n$.
		\item $x^m + x^n \approx x^{m+n} + 1$.
	\end{enumerate}
\end{lem}

Note that the first equation is a special case of the second for $m = n$, by invertibility of positive integers, but it nevertheless seems worth stating separately.

\begin{proof}
	By \Cref{power_lemma}.
\end{proof}

\subsection*{A supermodularity inequality in preordered semifields}

So far, most of our inequality results have been concerned with polynomial expressions involving only a single element of $F$. We now move beyond that case.

\begin{lem}
	\label{super_general}
	For $x \in F^\times$ with $x + x^{-1} \ge 2$ and any $a \in F$, the function
	\[
		\Z \longrightarrow F, \qquad n \longmapsto a + x^n
	\]
	is multiplicatively supermodular: for all $m, n \in \N$ and $\ell \in \Z$,
	\[
		(a + x^{\ell + m + n}) (a + x^\ell) \ge (a + x^{\ell + m}) (a + x^{\ell + n}).
	\]
	Moreover, this holds with $\approx$ under the stronger assumption $x + x^{-1} \approx 2$.
\end{lem}

\begin{proof}
	Since $a$ is arbitrary, it is sufficient to consider the case $\ell = 0$ for simplicity. Then we have
	\begin{align*}
		(a + x^{m+n})(a + 1)	& = a^2 + a(x^{m+n} + 1) + x^{m+n} \\
					& \ge a^2 + a(x^m + x^n) + x^{m+n} \\
					& = (a + x^n) (a + x^m),
	\end{align*}
	where the inequality step is by \Cref{power_lemma}\Cref{supermodular}.
	The claim about $\approx$ holds likewise by \Cref{power_square}.
\end{proof}

\subsection*{Towards cancellation criteria in preordered semifields}

The following results will later be strengthened, under additional hypotheses, to useful cancellation criteria. 

\begin{lem}
	\label{cancel1}
	Suppose that $x, y \in F^\times$ satisfy $x + x^{-1} \ge 2$ and $y \ge 1$. Then
	\[
		x + 1 \le y + 1 \quad \Longrightarrow \quad x^n \le y^{n+1} \enspace \forall n \in \N.
	\]
\end{lem}

\begin{proof}
	By \Cref{add_to_mult}, the assumption implies that for given $n \in \N$,
	\[
		\binom{n+1}{2} x^{n+1} + \sum_{j=0}^n x^j \le \binom{n+1}{2} y^{n+1} + \sum_{j=0}^n y^j.
	\]
	Together with \Cref{power_skew3}, we get
	\[
		\binom{n+2}{2} x^n \le \binom{n+1}{2} x^{n+1} + \sum_{j=0}^n x^j \le \binom{n+1}{2} y^{n+1} + \sum_{j=0}^n y^j \le \binom{n+2}{2} y^{n+1},
	\]
	where the final step is simply by using $y \ge 1$.
\end{proof}

\begin{lem}
	\label{cancel2}
	Suppose that $a, x, y \in F^\times$ satisfy $x + x^{-1} \ge 2$ and $y \ge 1$. Then
	\[
		\begin{matrix}
			a + x \le a + y \\[2pt]
			a^{-1} + x \le a^{-1} + y
		\end{matrix}
		\quad \Longrightarrow \quad x^n \le y^{n+1} \enspace \forall n \in \N.
	\]
\end{lem}

\begin{proof}
	Adding the first inequality of the main assumption to $a$ times the second gives
	\[
		(a + 1)(x + 1) = a + x + 1 + ax \le a + y + 1 + ay = (a + 1)(y + 1),
	\]
	so that the claim follows from the previous lemma upon cancelling $a + 1$.
\end{proof}

\section[Type classification of multiplicatively Archimedean fully preordered semifields]{Type classification of multiplicatively Archimedean\\ fully preordered semifields}
\label{malt_full}

\Cref{real_or_tropical} is an embedding theorem for multiplicatively Archimedean totally preordered semifields.
We now aim at generalizing this statement to a substantially more difficult case, namely from total semifield preorders to those that are merely full in the sense of \Cref{totalfull}, i.e.~total on connected components.
In particular, this does not require $0$ and $1$ to be ordered relative to one another.
We first explain why an embedding theorem as simple as \Cref{real_or_tropical} cannot be expected to hold.

\begin{ex}
	\label{dual_numbers_semifield}
	Let $F \coloneqq \R_{(+)}[X]/(X^2)$ be the semifield of all linear functions $r + s X$ with $r > 0$ or $r = s = 0$ modulo $X^2$. This is a semifield because $(r + s X)^{-1} = r^{-2}(r - s X)$ for $r > 0$, and it becomes a preordered semifield if we put
	\[
		r_1 + s_1 X \le r_2 + s_2 X \qquad : \Longleftrightarrow \qquad r_1 = r_2 \enspace \land \enspace s_1 \le s_2.
	\]
	It is clear that this preorder is full. However, this fully preordered semifield has the counterintuitive feature that $a \sim 1$ in $F$ implies $a + a^{-1} = 2$. Note that these are exactly the elements of the form $a = 1 + s X$ for any $s \in \R$.
	
	Moreover, since $x + x^{-1} = 2$ in $\R_+$ or $\TR_+$ only happens for $x = 1$, it follows that every monotone homomorphism $\phi : F \to \R_+$ or $\phi : F \to \TR_+$ satisfies $\phi(a) = 1$ for $a \sim 1$. Therefore there is no order embedding of $F$ into $\R_+$ or $\TR_+$.
\end{ex}

The following characterization of multiplicative Archimedeanicity will be slightly more convenient than the definition, and we will use it repeatedly.

\begin{lem}
	\label{arch_crit}
	A fully preordered semifield $F$ is multiplicatively Archimedean if and only if for all nonzero $x, y > 1$ there is $k \in \N$ with $y \le x^k$.
\end{lem}

\begin{proof}
	Suppose that $F$ is multiplicatively Archimedean, which means that $x^k \le y$ for all $k \in \N$ implies $x \le 1$ for all $x, y \in F^\times$.
	Then if $x, y > 1$, we in particular have $x \sim y^k$ for all $k \in \N$.
	Now if the claim is false, then we would have $x > y^k$ for all $k \in \N$, which by multiplicative Archimedeanicity would imply $x \le 1$, contradicting the assumption $x > 1$.

	Conversely if the condition in the claim holds, let $x, y \in F^\times$ be given such that $x^k \le y$ for all $k \in \N$.
	Then taking $k = 0$ and $k = 1$ gives $x \sim y \sim 1$, and hence $x \le 1$.
	If $x > 1$ instead of $x \le 1$, then there would be $k \in \N$ with $y \le x^k$ by assumption, contradicting $x^k \le y$ for all $k \in \N$.
\end{proof}

The following derived preorder will be a useful tool for the upcoming classification theory.

\begin{defn}
	\label{layer_preorder}
	Let $F$ be a totally preordered semifield, and let $a \in F^\times$ with $a \ge 1$ be given. Then the \newterm{layer preorder} $\le_a$ is the relation defined by
	\begin{equation}
		\label{layer_preorder_def}
		x \le_a y \quad :\Longleftrightarrow \quad
			\left( \exists k \in \N: x \le y a^k \: \land \: y \le x a^k \right) \:\land\:
			\left( \forall n \in \N: x^n \le y^n a \right)
	\end{equation}
\end{defn}

\begin{lem}
	The layer preorder $\le_a$ makes $F$ into a multiplicatively Archimedean fully preordered semifield.
\end{lem}

\begin{proof}
	To show this, we first prove transitivity of $\le_a$, where the nontrivial part is to show that $x^n \le y^n a$ and $y^n \le z^n a$ for all $n$ imply $x^n \le z^n a$. 
	This is because $x^n \sim y^n \sim z^n$ and $a \sim 1$ give $x^n \sim z^n a$, and $x^n > z^n a$ would imply $x^{2n} > z^{2n} a^2$ and therefore be in contradiction with
	\[
		x^{2n} \le y^{2n} a \le z^{2n} a^2.
	\]
	The $\le_a$-monotonicity of multiplication is obvious, while the $\le_a$-monotonicity of addition follows by $a \ge 1$ and the binomial expansion. Fullness follows by totality of $\le$ and $x \sim_a y$ being equivalent to the first condition involving $k$ only. It remains to establish multiplicative Archimedeanicity.
	So let $x, y >_a 1$.
	Then there are $k, n \in \N$ with $y \le a^k$ and $x^n > a$, and hence $y \le x^{kn}$, which produces the claim.
	So let $x, y \in F^\times$ be such that $x \le_a y^k$ for all $k \in \N$.
\end{proof}

\begin{ex}
	Consider the semifield $F$ of linear function $r + s X$ from \Cref{dual_numbers_semifield}, preordered with
	\[
		r_1 + s_1 X \le r_2 + s_2 X \qquad \Longleftrightarrow \qquad r_1 < r_2 \enspace \lor \enspace (r_1 = r_2 \enspace \land \enspace s_1 \le s_2).
	\]
	This makes $F$ into a totally preordered semifield.
	The layer preorder $\le_{1 + X}$ is then precisely the semifield preorder from~\Cref{dual_numbers_semifield}.
\end{ex}

Perhaps surprisingly, there is a way to associate real numbers to elements of multiplicatively Archimedean fully preordered semifields in such a way that the assigned numbers measure the ``size'' of the elements.

\begin{lem}
	\label{rate_lem}
	Fix nonzero $u > 1$ in a multiplicatively Archimedean fully preordered semifield $F$. Then for every nonzero $x \sim 1$ in $F$, there is a unique $r \in \R$ such that the following hold for all $\frac{p}{q} \in \Q$ with $q \in \Nplus$:
	\begin{enumerate}
		\item If $\frac{p}{q} < r$, then $x^q > u^p$.
		\item If $\frac{p}{q} > r$, then $x^q < u^p$.
	\end{enumerate}
\end{lem}

\begin{proof}
	It is clear that there can be at most one such $r$, since otherwise $x^q$ and $u^p$ would be strictly ordered in both directions for any $\frac{p}{q}$ that lies between them. 

	For any nonzero $y \sim 1$ and $t \in \N_{>0}$, we have $y \ge 1$ if and only if $y^t \ge 1$.
	Indeed $y \ge 1$ gives $y^t \ge 1$ by monotonicity of multiplication.
	If $y^t \ge 1$ instead, then $y < 1$ would be the only alternative to $y \ge 1$ by fullness, and this would give $y^t < y^{t-1} < \ldots < y < 1$, contradicting $y^t \ge 1$.

	Looking at $x$ and $u$ now, for any $p, q \in \Z$ we must have exactly one of
	\begin{equation}
		\label{trichotomy}
		x^q > u^p, \qquad x^q < u^p, \qquad x^q \approx u^p
	\end{equation}
	by fullness of the preorder and $x \sim 1 \sim u$.
	Applying the statement of the previous paragraph with $y \coloneqq x^q u^{-p}$ shows that which one of these cases happens is invariant under replacing $p$ and $q$ by $t p$ and $t q$ for any $t \in \N_{>0}$.
	Therefore the set of rationals
	\[
		D \coloneqq \left\{ \frac{p}{q} \in \Q \:\Big|\: x^q > u^p \right\}
	\]
	is well-defined.
	Using common denominators for any two given fractions $p/q > p'/q$ together with $u^p \ge u^{p'}$ shows that $D$ is downwards closed.
	Similarly, the set
	\[
		U \coloneqq \left\{ \frac{p}{q} \in \Q \:\Big|\: x^q < u^p \right\}
	\]
	is upwards closed.
	Moreover, if some $p/q$ is strictly smaller than all rationals in $D$, then it clearly must belong to $U$ and vice versa:
	again by the previous paragraph, the third alternative in~\eqref{trichotomy} can happen for at most one equivalence class of fractions.

	To see that $D$ and $U$ make up a Dedekind cut, it thus remains to show that they are both nonempty.
	But indeed $x^q < u^p$ holds for sufficiently large $\frac{p}{q}$ since $u$ is a power universal element, and similarly $x^q > u^p$ holds for small enough (sufficiently negative) $\frac{p}{q}$. 
\end{proof}

Before diving further into the classification theory, here is a useful criterion for deriving ordering relations.

\begin{lem}
	\label{rate_crit}
	Let $F$ be a multiplicatively Archimedean fully preordered semifield, and suppose that $x,y,z \in F^\times$ satisfy $x\sim y$ and $z \sim 1$. Suppose that for every $n \in \N$ we have $x^n \le y^n z$. Then $x \le y$.
\end{lem}

\begin{proof}
	Assume $x > y$ for contradiction. Then there is $k \in \N$ such that $z < (x y^{-1})^k$ by multiplicative Archimedeanicity, and hence $y^k z < x^k \le y^k z$, a contradiction.
\end{proof}

\subsection*{The five basic types and the type classification}

In \Cref{real_or_tropical}, we had distinguished the real and the tropical case as well as their two opposites. In our present more general context, it will be useful to distinguish five cases, where the additional case \ref{arctic_case} below corresponds to the situation of \Cref{dual_numbers_semifield}.

\begin{defn}
	\label{five_types_def}
	A preordered semifield $F$ is 
	\begin{enumerate}[label=(\roman*)]
		\item \newterm{max-tropical} if
			\[
				x + x^{-1} \approx 2x,
			\]
		\item \newterm{max-temperate} if
			\[
				2 < x + x^{-1} < 2x,
			\]
		\item\label{arctic_case} \newterm{arctic} if
			\[
				x + x^{-1} \approx 2,
			\]
		\item \newterm{min-temperate} if
			\[
				2x^{-1} < x + x^{-1} < 2,
			\]
		\item \newterm{min-tropical} if
			\[
				x + x^{-1} \approx 2x^{-1},
			\]
	\end{enumerate}
	holds for all $x \in F^\times$ with $x > 1$.
	
	We say that $F$ is \newterm{tropical} if it is min-tropical or max-tropical, and similarly \newterm{temperate} if it is min-temperate or max-temperate.
\end{defn}

For any $F$ and any nonzero $x > 1$, the element $x + x^{-1}$ must lie somewhere in the order interval $[2x^{-1}, 2x]$. These five types thus make a distinction depending on where $x + x^{-1}$ lands in that interval, using the three elements
\[
	2x^{-1} < 2 < 2x
\]
for comparison.
Since the answer is required to be the same for all $x > 1$, it follows that if $F$ is any preordered semifield for which the preorder on $F^\times$ is nontrivial, then it can be of at most one of these five types. 

\begin{ex}
	$\R_+$ is max-temperate and $\TR_+$ is max-tropical.
	Similarly, $\R_+^\op$ is min-temperate and $\TR_+^\op$ is min-tropical.\footnote{More generally, reversing the preorder from $F$ to $F^\op$ also ``reverses'' the type.}
	\Cref{dual_numbers_semifield} is arctic, and is isomorphic to its own opposite via $r + s X \mapsto r - s X$.
	Hence all five types do occur.

	A general preordered semifield does not need to be of either type. For example, taking the categorical product of two preordered semifields of different types produces a preordered semifield which does not have a type.
\end{ex}

\begin{rem}
	Our choice of terminology \emph{tropical}, \emph{temperate} and \emph{arctic} is based on the historical contingency of the established term \emph{tropical}. Arguably an intrinsically more reasonable choice would be to use \emph{tropical} and \emph{arctic} in the exactly opposite manner, for two reasons. First, there are two tropical cases but only one arctic case, and the latter is sandwiched in between the two temperate cases.
	This is exactly opposite to how the tropics are sandwiched in between the arctic regions in geographical reality. Second, in terms of an analogy with thermodynamics, the tropical cases correspond to zero temperature~\cite{tropical}, which is indeed rather cold; in contrast to this, we conjecture that the arctic case can be associated with infinite temperature. 
\end{rem}

\begin{rem}
	While we will not do this in the present work, it may also be of interest to refine the above definition so as to assign a type to any strictly ordered pair of elements $x < y$ in any preordered semiring $S$, by similarly considering where the element $x^2 + y^2$ falls relative to the elements
	\[
		2 x^2 < 2 x y < 2 y^2.
	\]
	Here, $x < y$ implies that $x^2 + y^2$ must be somewhere in the order interval $[2x^2, 2y^2]$.
\end{rem}

Here is why the five types are relevant in our context.

\begin{prop}
	\label{three_cases}
	Suppose that $F$ is a multiplicatively Archimedean fully preordered semifield with nontrivial preorder on $F^\times$. Then $F$ is of exactly of one of the five types of \Cref{five_types_def}.
\end{prop}

\begin{proof}
	By the nontriviality assumption, there must be \emph{some} $x > 1$ in $F^\times$, implying that the five conditions are mutually exclusive. We therefore only need to show that if some fixed nonzero $x > 1$ satisfies one of them, then any other nonzero $y > 1$ also satisfies exactly the same condition. To do so, we use multiplicative Archimedeanicity to choose $k \in \N$ with $y < x^k$.

	First, we tackle the max-tropical case by showing that $x + x^{-1} \approx 2x$ implies $y + y^{-1} \approx 2y$. Since $y + y^{-1} \le 2y$ is trivial by $y > 1$, it is enough to prove $y + y^{-1} \ge 2y$. We have
	\[
		2(x^2 + x^{-2}) \approx x^2 + 1 + 2 x^{-2} \approx x^2 + 2 + x^{-2} \approx (x + x^{-1})^2 \approx 4 x^2,
	\]
	and hence by iteration $x^k + x^{-k} \approx 2x^k$ whenever $k$ is a power of two, which we can assume for the above $k$ without loss of generality.
	Thus upon replacing $x$ by $x^k$, we can also assume $y < x$. But then
	\[
		y + y^{-1} = y(1 + y^{-2}) \ge y(1 + x^{-2}) \approx 2y,
	\]
	as was to be shown.

	Second, for the max-temperate case, it is thus enough to show that the inequality $x + x^{-1} > 2$ implies $y + y^{-1} > 2$. We have $x^k + x^{-k} > 2$ by \Cref{norder}, so that replacing $x$ by $x^k$ lets us again assume $y < x$ without loss of generality. By \Cref{power_lemma}\ref{exponential_better} and multiplicative Archimedeanicity, we can find $n \in \N$ such that
	\[
		2^{n-1} (x^n + x^{-n}) \ge (x + x^{-1})^n > 2^n x,
	\]
	resulting in $x^n + x^{-n} > 2 x$. By $x > 1$, we can weaken this to $x^n + x^{-(n+1)} > 2$. Now let $m \in \N$ be the smallest integer with $x^n \le y^m$. Then also $y^m < x^{n+1}$, since $x^{n+1} \le y^m$ would contradict the minimality of $m$ by $y < x$.
	This gives
	\[
		y^m + y^{-m} \ge x^n + x^{-(n+1)} > 2.
	\]
	Since $y + y^{-1} \le 2$ would imply $y^m + y^{-m} \le 2$ again by \Cref{norder}, this proves that indeed $y + y^{-1} > 2$ due to fullness.

	Third, $x + x^{-1} \approx 2$ implies $y + y^{-1} \approx 2$. For since $F$ is fully preordered, $y > 1$ implies that $y + y^{-1} \le 2$ or $y + y^{-1} \ge 2$ or both. It is indeed both, since by the previous paragraph a strict inequality would also imply a strict inequality between $y + y^{-1}$ and $2$.

	The other cases follow by symmetry upon replacing $F$ by $F^\op$.
\end{proof}

\subsection*{A cancellation criterion}

Throughout this subsection and the following ones, $F$ is still a multiplicatively Archimedean fully preordered semifield.

Thanks to multiplicative Archimedeanicity, we can improve on the inequalities derived in \Cref{intermezzo}.
In particular, we can turn \Cref{cancel2} into an actual cancellation criterion.

\begin{prop}
	\label{cancel2plus}
	Let $a,x,y \in F^\times$ with $x \sim 1$ be such that $x + x^{-1} \ge 2$ and $y \ge 1$. Then
	\[
		\begin{matrix}
			a + x \le a + y \\[4pt]
			a^{-1} + x \le a^{-1} + y
		\end{matrix}
		\quad \Longrightarrow \quad x \le y.
	\]
\end{prop}

\begin{proof}
	Combine \Cref{cancel2} with \Cref{rate_crit}.
\end{proof}

Over the course of the next few short subsections, we will sharpen the type classification by deriving further inequalities for $F$ under type hypotheses.

\subsection*{The max-tropical case}

The following justifies the term ``max-tropical'' further by clarifying in what sense addition on max-tropical $F$ is analogous to addition in the tropical semifield $\TR_+$.

\begin{lem}
	\label{tropical_add_full}
	Let $F$ be max-tropical and $x,y \in F^\times$. If $x\sim y$, then $x + y \approx 2 \max(x,y)$.
\end{lem}

Note that this obviously holds in $\TR_+$, where we have $x + y = \max(x,y)$ and $2 = 1$.

\begin{proof}
	We can assume $x > 1$ and $y = 1$ without loss of generality, in which case we need to show $x + 1 \approx 2x$. We already know $x^2 + 1 \approx 2x^2$ by max-tropicality, and hence
	\[
		(x + 1)^2 = x^2 + 1 + 2x \approx 2x^2 + 2x = 2x (x + 1),
	\]
	which implies the claim by invertibility of $x + 1$.
\end{proof}

Of course, if $F$ is min-tropical, then we similarly get $x + y \approx 2 \min(x,y)$ for $x \sim y$.

\subsection*{The arctic case}

Something analogous works in the arctic case. It is an instructive exercise to verify the following explicitly for \Cref{dual_numbers_semifield}.

\begin{lem}
	\label{arctic_main}
	Let $F$ be arctic and $x,y \in F^\times$. If $x,y \sim 1$, then also
	\[
		x + y \approx xy + 1.
	\]
\end{lem}

\begin{proof}
	We first show that the equation holds with $x^2$ and $y^2$ in place of $x$ and $y$, in which case
	\[
		x^2 + y^2 = xy (xy^{-1} + (xy^{-1})^{-1}) \approx 2 xy \approx x^2 y^2 + 1
	\]
	proves the claim. This gives the general case via
	\[
		(x + y)^2 = x^2 + y^2 + 2xy \approx x^2 y^2 + 2xy + 1 = (xy + 1)^2,
	\]
	since the squares can be cancelled: a strict inequality in either direction would likewise hold for their squares.
\end{proof}

Given a polynomial or Laurent polynomial $p$, we write $p'$ for its derivative and obtain the following formulas for evaluating $p$.

\begin{lem}
	\label{arctic_formula}
	Suppose that $F$ is arctic and let $x \in F^\times$ with $x \sim 1$ and $p, q \in \N[X, X^{-1}]$ nonzero. 
	Then:
	\begin{enumerate}[label=(\roman*)]
		\item\label{arctic_formula_i} $p(x) \approx x^{p'(1)} + (p(1) - 1)$.
		\item\label{arctic_formula_iii} If $x > 1$, then $p(x) \le q(x)$ if and only if $p(1) = q(1)$ and $p'(1) \le q'(1)$.
		\item\label{arctic_formula_iv} If $x \not\approx 1$, then $p(x) \approx q(x)$ if and only if $p(1) = q(1)$ and $p'(1) = q'(1)$.
	\end{enumerate}
\end{lem}

It is instructive to consider how these formulas manifest themselves in the case of \Cref{dual_numbers_semifield}.

\begin{proof}
	For \ref{arctic_formula_i}, we use induction on the sum of coefficients $p(1)$, where the base case $p(1) = 1$ is trivial since then $p$ is necessarily a single monomial.
	For the induction step, we write $p = x^{\deg(p)} + \hat{p}$ for some $\hat{p} \in \N[X, X^{-1}]$, and obtain by the induction assumption
	\[
		p(x) \approx x^{\hat{p}'(1)} + (\hat{p}(1) - 1) + x^{\deg(p)} \approx x^{\hat{p}'(1) + \deg(p)} + \hat{p}(1) \approx x^{p'(1)} + (p(1) - 1),
	\]
	where the second step is by \Cref{arctic_main}.

	For \ref{arctic_formula_iii}, the ``if'' part is immediate from \ref{arctic_formula_i}.
	For the ``only if'' part, we assume $p(x) \le q(x)$.
	We first derive $p(1) = q(1)$ by showing that $p(1) \neq q(1)$ leads to a contradiction.
	We have
	\[
		p(1) \sim p(x) \sim q(x) \sim q(1).
	\]
	In terms of $r \coloneqq p(1)$ and $s \coloneqq q(1)$, we hence obtain $r \sim s$ although $r \neq s$.
	Then since $F$ is arctic, we get $\frac{r}{s} + \frac{s}{r} \approx 2$.
	Thus $t \coloneqq \frac{1}{2}\left(\frac{r}{s} + \frac{s}{r}\right)$ is a positive rational in $(1,\infty)$ with $t \approx 1$.
	Since some power $t^n$ exceeds $2$ in the usual order on $\Q_+$, we can write
	\[
		2 = \frac{t^n + w}{1 + w}
	\]
	for a unique $w \in \Q_{>0}$ and then also conclude $2 \approx 1$.
	But then
	\[
		3 x \approx 2 + x^3 \approx 1 + 2 x^3 \approx 3 x^2,
	\]
	which contradicts the assumed $x > 1$.
	Therefore indeed $p(1) = q(1)$.

	Now consider again $p(x) \le q(x)$.
	Since $p(1) = q(1)$, we obtain
	\[
		x^{p'(1)} + (p(1) - 1) \le x^{q'(1)} + (p(1) - 1)
	\]
	by \ref{arctic_formula_i}.
	This implies $x^{p'(1)} + 1 \le x^{q'(1)} + 1$ by chaining.
	But then $x^{p'(1)} \le x^{q'(1)}$ because of \Cref{cancel2plus}, which implies the missing $p'(1) \le q'(1)$ thanks to $x > 1$.
	
	Finally, \ref{arctic_formula_iv} follows directly from \ref{arctic_formula_iii}.
\end{proof}

\subsection*{Away from the tropical case}

If $F$ is max-tropical, then we have $x^{-n} + 1 \approx 2$ for every nonzero $x > 1$ and $n \in \N$. If $F$ is min-tropical, then we similarly have $x^n + 1 \approx 2$. The following result can be thought of as providing converse statements.

\begin{lem}
	\label{nontropical_cancel}
	Let $x \in F^\times$ with $x > 1$.
	\begin{enumerate}
		\item If $F$ is not min-tropical, then
			\[
				x^n + 1 < x^{n+1} + 1
			\]
			for all $n \in \N$.
		\item If $F$ is not max-tropical, then
			\[
				x^{-(n+1)} + 1 < x^{-n} + 1
			\]
			for all $n \in \N$.
	\end{enumerate}
	Thus if $F$ is not tropical, then the map $n \mapsto x^n + 1$ is strictly increasing across all $n \in \Z$.
\end{lem}

\begin{proof}
	These two cases become equivalent upon reversing the order and replacing $x$ by $x^{-1}$. We therefore only treat the first case.
	
	If we had $x^{n+1} + 1 \le x^n + 1$ for some $n$, then we would get $(x + 1)^m \le 2^m x^n$ for all $m \in \N$ by \Cref{other_power_lemma} and therefore $x + 1 \le 2$ by \Cref{rate_crit}. But then also
	\[
		x^2 + 3 \le x^2 + 2 x + 1 = (x + 1)^2 \le 4,
	\]
	so that chaining gives $3(x^2 + 1) \le 6$, or equivalently $x + x^{-1} \le 2 x^{-1}$, which contradicts the assumption that $F$ is not min-tropical.
\end{proof}

\begin{lem}
	\label{power_lemma_nontropical}
	Suppose that $F$ is not max-tropical and let $x \in F^\times$. If $x > 1$, then for every $n \in \Nplus$ there is $k \in \N$ such that
	\[
		x^n + k \le (k + 1) x.
	\]
\end{lem}

\begin{proof}
	We show this first for $n = 2$. If $x + x^{-1} \le 2$, then this holds with $k = 1$, so assume $x + x^{-1} > 2$, meaning that $F$ is max-temperate.
	
	Using $x + 1 < 2 x$ from \Cref{nontropical_cancel}, \Cref{rate_crit} shows that there is $m \ge 2$ such that
	\[
		(x + 1)^{m+1} \le 2^{m+1} x^m.
	\]
	Expanding the left-hand side and using $x \ge 1$ gives the weaker bound
	\[
		x^{m+1} + (2^{m+1} - 1) \le 2^{m+1} x^m.
	\]
	Thus there are $m,k \in \N$ such that
	\beq
		\label{mestimate}
		x^{m+1} + k \le (k + 1) x^m.
	\eeq
	We now claim that if this holds for some $m \ge 2$, then it also holds with $m - 1$ in place of $m$. Indeed the following estimates show that it is enough to increase $k$ by $1$,
	\begin{align*}
		(x + 1) (x^m + (k + 1))		& = x^{m+1} + x^m + (k + 1) x + k + 1 \\
						& \le (k + 2) x^m + (k + 1) x + 1 \\
						& \le (k + 2) x^m + (k + 2) x^{m-1} \\
						& = (x + 1) (k + 2) x^{m-1},
	\end{align*}
	where the first inequality step uses the assumption, and the second one uses merely $x \ge 1$ and $m \ge 2$. Upon iterating this argument, we therefore conclude that \eqref{mestimate} holds even with $m = 1$, meaning that there is $k$ such that
	\[
		x^2 + k \le (k + 1) x,
	\]
	as was to be shown for $n = 2$.

	We now show that if the claim holds for $n \ge 2$, then it also holds for $2n$,
	\begin{align*}
		x^{2n} + (k + 1)^3	& \le x^{2n} + 2 k x^n + k^2 + (k^3 + 2 k^2 + k + 1) \\
					& = (x^n + k)^2 + (k^3 + 2 k^2 + k + 1) \\
					& \le (k + 1)^2 x^2 + (k^3 + 2 k^2 + k + 1) \\
					& = (k + 1)^2 \left(x^2 + k \right) + 1 \\
					& \le \left( (k + 1)^3 + 1 \right) x,
	\end{align*}
	where the first inequality step uses only $x \ge 1$ and the other two use the assumption.
	In particular the claim holds whenever $n$ is a power of two. This is enough for the general case by the monotonicity in $n$ proven just before.
\end{proof}

We also derive a further statement which makes explicit use of positive linear combinations with rational coefficients. Recall that these exist in any strict semifield.

\begin{lem}
	\label{nottrop_convex}
	Suppose that $F$ is not tropical and $x \in F^\times$. If $x > 1$, then for every rational $r \in (0,1)$, we have
	\[
		1 < r x + (1 - r) < x.
	\]
\end{lem}

\begin{proof}
	By reversing the order and replacing $x$ by $x^{-1}$ and $r$ by $1-r$, the second inequality reduces to the first.
	We therefore only prove the first. 

	Since the expression $r x + (1 - r)$ is obviously non-strictly monotone in $r$, it is enough to prove the claim for $r = 2^{-n}$ with $n \in \Nplus$, in which case it amounts to
	\[
		1 < 2^{-n} x + (1 - 2^{-n}).
	\]
	We indeed have $1 + 1 < 1 + x$ by \Cref{nontropical_cancel}, which is the $n = 1$ case. For the induction step from $n$ to $n + 1$, we apply this same inequality $1 + 1 < 1 + y$ with $y \coloneqq 2^{-n} x + (1 - 2^{-n})$, which satisfies $y > 1$ by the induction assumption.
\end{proof}

\subsection*{Away from the arctic case}

While the previous lemmas were concerned with $F$ not being max-tropical or min-tropical, we now consider a similar statement for $F$ not arctic.

\begin{lem}
	\label{nonarctic_bound_concrete}
	Suppose that $F$ is not arctic and $x \in F^\times$. If $x > 1$ and $x + x^{-1} > 2$, then for every $\ell \in \N$ there is $n \in \N$ such that for all $m \in \N$.
	\[
		x^{n+1} + \ell x^{-m} > x^n + \ell.
	\]
\end{lem}

While this quite clear in the max-tropical case, the main difficulty lies in proving it in the max-temperate case (but restricting to this case explicitly would not simplify the proof).

\begin{proof}
	We prove a number of auxiliary statements first before getting to the claim itself.
	\begin{enumerate}
		\item\label{epspositive} There is $n \in \N$ such that $x^n + x^{-n} \ge 2x$.
			
			Indeed by \Cref{rate_crit}, there is $n \in \N$ such that $(x + x^{-1})^n \ge 2^n x$. Hence by \Cref{power_lemma},
			\[
				2^{n-1} (x^n + x^{-n}) \ge (x + x^{-1})^n \ge 2^n x,
			\]
			as was to be shown.
		\item\label{epsone} For every $\eps < 1$ in $\R$ there are $m,n \in \Nplus$ such that $m > \eps n$ and
			\beq
				\label{super_dominance}
				x^n + x^{-n} \ge 2 x^m.
			\eeq
			
			Indeed if this inequality holds for given $n$ and $m$, then it also holds for all multiples, since for every $\ell \in \Nplus$,
			\[
				x^{\ell n} + x^{-\ell n} \ge 2^{1-\ell } (x^n + x^{-n})^\ell \ge 2 x^{\ell m},
			\]
			where the first step is by \Cref{power_lemma} and the second by assumption. Now let $\eps$ be the supremum of all fractions $\frac{m}{n}$ for which the inequality~\eqref{super_dominance} holds; our goal is to show that $\eps = 1$, whereas what we know by \ref{epspositive} is $\eps > 0$. Indeed we claim that $\eps \ge \frac{3 - \eps^2}{2} \eps$,
			which then implies $\eps = 1$ because of $0 < \eps \le 1$. In order to prove this claim, suppose that a fraction $\frac{m}{n}$ satisfies the inequality. Then also
			\begin{align*}
				2(x^{2n^3} + x^{-2n^3}) & \ge (x^{2n^3} + x^{2mn^2}) + (x^{2mn^2} + x^{-2n^3}) \\[2pt]
							& = x^{n^2(n+m)} (x^{n^2(n-m)} + x^{-n^2(n-m)}) + x^{-n^2(n-m)} (x^{n^2(n+m)} + x^{-n^2(n+m)}) \\[2pt]
							& \ge 2 \left( x^{n^2(n+m)} x^{mn(n-m)} + x^{-n^2(n-m)} x^{mn(n+m)} \right) \\[2pt]
							& = 2 \left( x^{n(n^2 + 2mn - m^2)} + x^{-n(n^2 - 2mn - m^2)} \right) \\[2pt]
							& = 2 x^{2mn^2} (x^{n(n^2 - m^2)} + x^{-n(n^2 - m^2)}) \\[2pt]
							& \ge 4 x^{2mn^2} x^{m(n^2 - m^2)} \\[2pt]
							& = 4 x^{3mn^2 - m^3}
			\end{align*}
			where all inequality steps are per the above. Therefore
			\[
				\eps \ge \frac{3mn^2 - m^3}{2n^3} = \frac{3-\left(\frac{m}{n}\right)^2}{2} \cdot \frac{m}{n} 
			\]
			Thus as $\frac{m}{n} \nearrow \eps$, we get the claimed $\eps \ge \frac{3 - \eps^2}{2} \, \eps$.
		\item\label{ellone} There is $n \in \N$ such that
			\[
				x^n + x^{-m} \ge 2
			\]
			for all $m \in \N$.

			Taking $\eps = \frac{1}{2}$ in \Cref{epsone}, we have $n$ such that
			\[
					x^{2n} + x^{-2n} \ge 2 x^n.
			\]
			There is no loss in replacing $x$ by $x^n$, so that we can assume $x^4 + 1 \ge 2 x^3$ without loss of generality. But then also
			\[
				x^4 + 2 \ge 2 x^3 + 1 \ge 3 x^2,
			\]
			where the second step is by \Cref{power_lemma}. Therefore $x^2$ satisfies the hypotheses of \Cref{nonarctic_bound}, and we get that
			\[
				x^{2(m + 1)} + 1 \ge 2 x^{2m} 
			\]
			for all $m \in \N$. Therefore also
			\[
				x^2 + x^{-2m} \ge 2
			\]
			for all $m \in \N$, which is enough.
		\item\label{nzero} For every $\ell \in \N$ there is $n \in \N$ such that
			\[
				x^n + \ell x^{-m} \ge 1 + \ell
			\]
			for all $m \in \N$.
		
			Indeed for $\ell = 1$, this is exactly \ref{ellone}. Moreover if the inequality holds for some $\ell$, then it also holds for all $\ell' < \ell$, since multiplying the inequality by $\ell'$ and adding $\ell - \ell'$ times the inequality $x^n \geq 1$ results in
			\[
				\ell x^n + \ell \ell' x^{-m} \ge \ell + \ell \ell',
			\]
			which is equivalent to the desired inequality with $\ell'$ in place of $\ell$. Therefore it is enough to show that if the statement holds for given $\ell\in\N$, then it also holds for $2\ell + 1$.

			This step from $\ell$ to $2\ell + 1$ works as follows,
			\begin{align*}
				x^{2n} + (2 \ell + 1) x^{-m} & = x^n(x^n + \ell x^{-m-n}) + (\ell + 1) x^{-m} \\
				& \ge (\ell + 1) x^n + (\ell + 1) x^{-m} \\
				& = (\ell + 1) (x^n + x^{-m}) \\
				& \ge 2 (\ell + 1) ,
			\end{align*}
			where we have assumed that the given $n$ is large enough to work both for the given $\ell$ and for $\ell = 1$. 
		\item The actual claim is then the $k = 1$ case of the following: for every $k \in \Nplus$ and $\ell \in \N$ there is $n \in \N$ such that
			\[
				x^{n+k} + \ell x^{-m} \ge x^n + \ell
			\]
			for all $m \in \N$.

			Indeed \ref{nzero} shows that this holds for some $k$ with $n = 0$. Since it automatically holds for all larger $k$, it is enough to show that if the statement holds for a given even $k$, then it also holds with $\frac{k}{2}$ in place of $k$. Assuming $k$ to be even without loss of generality and replacing $x$ by $x^k$, it is enough\footnote{Note that $x^k + x^{-k} > 2$ by \Cref{power_lemma}.} to show that the $k = 2$ case implies the $k = 1$ case, at the cost of replacing $\ell$ by $2\ell$ and $n$ by $n + 3$,
			\begin{align*}
				(x + x^{-1}) (x^{n+4} + \ell x^{-m})	& = x^{n+5} + x^{n+3} + \ell x^{-m+1} + \ell x^{-m-1} \\
									& \ge x^{n+5} + x(x^{n+2} + 2 \ell x^{-m-2}) \\
									& \ge x^{n+5} + x(x^n + 2 \ell) \\
									& \ge x^{n+4} + x^{n+2} + 2 \ell x \\
									& \ge (x + x^{-1}) (x^{n+3} + \ell),
			\end{align*}
			where the first and fourth inequality step use merely $x \ge 1$, the second is by assumption, and the third by $x^2 + x^{-2} \ge x + x^{-1}$ from \Cref{power_lemma}.\qedhere
	\end{enumerate}
\end{proof}

It may also be of interest to know under which conditions a semifield $F$ can support any nontrivial full semifield preorder at all that is multiplicatively Archimedean and arctic.
The following result provides one relevant criterion.

\begin{prop}
	\label{no_arctic}
	Let $F$ be a strict semifield with quasi-complements such that the ring $F \otimes \Z$ is absolutely flat\footnote{Recall that a ring $R$ is \newterm{absolutely flat} if every ideal in $R$ is idempotent. For example, every product of fields is absolutely flat.}. Then every multiplicatively Archimedean full semifield preorder on $F$ is temperate or tropical.
\end{prop}

\begin{proof}
	Let $\le$ be a multiplicatively Archimedean full semifield preorder on $F$.
	Then
	\[
		I \coloneqq \{ x - y \mid x \sim y \textrm{ in } F\}
	\]
	is an ideal in $F \otimes \Z$, which is idempotent by assumption. Therefore for $a > 1$, there are elements $x_i, y_i, c_i \in F$ for $i=1,\ldots,\ell$ such that $x_i, y_i \sim 1$ and
	\[
		a - 1 = \sum_{i=1}^\ell c_i (x_i - 1) (y_i - 1)
	\]
	holds in $F \otimes \Z$. But this means that there is $d \in F$ such that
	\[
		a + \sum_{i=1}^\ell c_i (x_i + y_i) + d = 1 + \sum_{i=1}^\ell c_i (x_i y_i + 1) + d.
	\]
	Now if $\le$ was arctic, then we would have $x_i + y_i \sim x_i y_i + 1$ by \Cref{arctic_main}, and therefore $a + z \approx 1 + z$ with $z \coloneqq \sum_{i=1}^\ell c_i (x_i y_i + 1) + d$.
	Upon adding a quasi-complement of $z$ on both sides, we obtain further $a + n \approx 1 + n$ for some $n \in \N$. By chaining, we can reduce to the case $n = 1$ without loss of generality.
	But then applying the cancellation criterion of \Cref{cancel2plus} to $1 + a \le 1 + 1$ shows $a \le 1$, contradicting the initial assumption $a > 1$.
\end{proof}

For example, the categorical product of $\R_+$ with itself (any number of times) is a semifield that satisfies the assumptions, and therefore does not support any multiplicatively Archimedean full semifield preorder of arctic type.

\section{The ambient preorder}
\label{ambient}

Perhaps surprisingly, every multiplicatively Archimedean fully preordered semifield can be equipped with a canonical \emph{total} semifield preorder which extends the given preorder, and often does so in such a way that this induces an order embedding into some $\mathbb{K} \in \{\R_+, \R_+^\op, \TR_+, \TR_+^\op\}$.
This derived preorder is defined as follows, for preordered semifields in general.

\begin{defn}
	\label{ambient_defn}
	Let $F$ be a preordered semifield. Given fixed elements $a, b \in F$, the \newterm{ambient preorder} $\le\le_{a,b}$ is the relation on $F$ defined by
	\[
		x \le\le_{a,b} y \quad : \Longleftrightarrow \quad a y + b x \le a x + b y .
	\]
\end{defn}

\begin{lem}
	If $a \not\approx b$, then the ambient preorder $\le\le_{a,b}$ also makes $F$ into a preordered semifield, and $1 \ge\ge_{a,b} 0$ if and only if $a \le b$.
\end{lem}

\begin{proof}
	The condition $a \not\approx b$ clearly guarantees $1 \,\,\cancel{\approx\approx}\,\, 0$.
	All other required properties are also straightforward to verify, apart from the transitivity of $\le\le_{a,b}$. The latter is where the assumption that $F$ is a semifield (rather than a mere semiring) comes in.
	Indeed assuming $x \le\le_{a,b} y \le\le_{a,b} z$, we have
	\[
		a y + b x \le a x + b y, \qquad a z + b y \le a y + b z.
	\]
	We then obtain
	\begin{align*}
		(x + y) (a z + b x)	& = b x^2 + (a z + b y) x + a y z \\
					& \le b x^2 + (a y + b z) x + a y z \\
					& = (x + z) (a y + b x) \\
					& \le (x + z) (a x + b y) \\
					& = a x^2 + (a z + b y) x + b y z \\
					& \le a x^2 + (a y + b z) x + b y z \\
					& = (x + y) (a x + b z) .
	\end{align*}
	Thus if $x + y \neq 0$, then the desired $x \le\le_{a,b} z$ follows. The complementary case is $x = y = 0$, in which case the claim holds trivially by $x = y$.
\end{proof}

\begin{ex}
	\label{56ex}
	The same definition of ambient preorder does not extend to general preordered semirings, since the transitivity may fail. For an explicit example, consider the semiring $S \coloneqq \N / (5 \simeq 6)$, which is $\N$ with all numbers $\ge 5$ identified with $5$. Equip $S$ with either the trivial preorder or the total preorder inherited from $\N$. In either case, we have $2 \le\le_{1,2} 5 \le\le_{1,2} 1$ but $2 \cancel{\le\le}_{1,2} 1$.
\end{ex}

\begin{rem}
	An interesting feature of the ambient preorder is its behaviour under reversing $\le$: we have $x \le\le_{a,b} y$ in $F$ if and only if $x \le\le_{b,a} y$ in $F^\op$.
\end{rem}

While the ambient preorder makes sense on any preordered semifield, we now return to the assumption that $F$ is a multiplicatively Archimedean fully preordered semifield, where the ambient preorder will facilitate the proof of our separation results.
Given such an $F$, we fix an arbitrary $u \in F^\times$ with $u > 1$. As the notation indicates, $u$ is power universal (by the definition of multiplicative Archimedeanicity). We suspect that the ambient preorder $\le\le_{1,u}$ is independent of the particular choice of $u$, but we have not been able to prove this so far, and we will not need it in the following. We nevertheless suppress the dependence on $u$ from our notation of the ambient preorder by writing $\le\le$ as shorthand for $\le\le_{1,u}$.
In other words, we put
\[
	x \le\le y \quad : \Longleftrightarrow \quad x u + y \le x + y u,
\]
and this is what we will use in the rest of this section.
By \Cref{nonarctic_bound_concrete}, we can find $n \in \N$ such that $u^{n+1} + u^{-m} \ge u^n + 1$ for all $m \in \N$.
Hence upon replacing $u$ by $u^{n+1}$ if necessary, we can achieve in particular that
\begin{equation}
	\label{ambient_power_universal}
	u + u^{-m} \ge 2
\end{equation}
for all $m \in \N$.
We assume from now on that such $u$ has been fixed.

\begin{lem}
	$\le\le$ is a total preorder.
\end{lem}

\begin{proof}
	Since $u \sim 1$, we have $x u + y \sim x + y u$. Hence this follows from the assumption that the preorder on $F$ is full.
\end{proof}

The following auxiliary results will play a key technical role in the proofs of our main theorems.

\begin{lem}
	\label{ambient_lemma}
	Let $F$ be a multiplicatively Archimedean fully preordered semifield of arctic, max-temperate or max-tropical type and with $u > 1$.
	Suppose that the quotient semifield $F/\!\sim$ has quasi-complements.
	Then the following holds for a suitable choice of $u$:
	\begin{enumerate}[label=(\roman*)]
		\item\label{ambient_extend} $\le\le$ extends $\le$.
		\item\label{ambient_embed} If $\le$ is max-temperate or max-tropical, then for all $x \sim y$ in $F$, we have
			\[
				x \le y \quad \Longleftrightarrow \quad x \le\le y,
			\]
			and every $x > 1$ is power universal with respect to $\le\le$.
	\end{enumerate}
	Furthermore, with $v \coloneqq 2u$ we have:
	\begin{enumerate}[label=(\roman*),resume]
		\item\label{ambient_power} $v$ is a power universal element for $\le\le$.
		\item\label{ambient_typegen} $v + v^{-1} > > 2$.
		\item\label{ambient_typearctic} If $\le$ is arctic, then also $v + v^{-1} < < 2 v$.
	\end{enumerate}
\end{lem}

Since reversing the original preorder $\le$ keeps the ambient preorder $\le\le$ invariant, these statements hold similarly in the min-temperate and min-tropical cases, where reversing the preorder also entails that $u$ needs to be replaced by $u^{-1}$.

\begin{proof}
	First, $u \ge 1$ shows that $1 \ge\ge 0$.
	\begin{enumerate}[label=(\roman*)]
		\item[(i), (ii)] We need to show that for all $x, y \in F$,
			\[
				x \le y \quad \Longrightarrow \quad x \le\le y,
			\]
			as well as the converse in the max-tropical and max-temperate cases, assuming that $x \sim y$.

			All of this is trivial if $x = y = 0$. If exactly one is nonzero, then for $x \sim y$ to hold we would need to have $1 \sim 0$, making $F$ totally preordered, in which case $F$ or $F^\op$ embeds into $\R_+$ or $\TR_+$ by \Cref{real_or_tropical}, where the claims follow by a straightforward computation. We thus assume that $x,y \in F^\times$, and put $y = 1$ without loss of generality.

			In the arctic case, we thus are assuming $x \le 1$ and need to show that
			\[
				x u + 1 \le x + u.
			\]
			This holds even with $\approx$ by \Cref{arctic_main}. 
			We thus turn to the max-temperate and max-tropical cases.
			If $x < 1$, then we can find $k \in \N$ such that $x^k u < 1$. Then
			\[
				x^k u + 1 \le 2 \le x^k + u,
			\]
			where the second inequality holds by power universality of $u$ and the assumed~\eqref{ambient_power_universal}. The first inequality is strict by \Cref{nontropical_cancel} in the max-temperate case, while the second inequality is strict in the max-tropical case by \Cref{tropical_add_full}. Thus $x^k u + 1 < x^k + u$ in both cases, and we conclude $x^k < < 1$. This implies $x < < 1$ by totality of $\le\le$.
			
			The case $x > 1$ is analogous, resulting in $x > > 1$. And finally if $x \approx 1$, then of course we also have
			\[
				x u + 1 \approx u + 1 = 1 + u \approx x + u,
			\]
			resulting in $x \approx\approx 1$.
			
			The final claim on power universality of $x > 1$ holds because some power of $x$ dominates $u$ since $F$ is multiplicatively Archimedean, and we will prove $u$ to be power universal in the upcoming proof of \ref{ambient_power}. 
			
		\setcounter{enumi}{2}
		\item The definition of the ambient preorder shows that $u \ge\ge 1$ is equivalent to $u + u^{-1} \ge 2$, which we have assumed. Since $2 \ge\ge 1$, this implies $v \ge\ge u \ge\ge 1$.

			For power universality, suppose first that $\le$ is not arctic, and therefore is max-temperate or max-tropical. We then show that even $u$ itself is a power universal element for $\le\le$, which means that for all $x \in F^\times$ there is $n \in \N$ with $x \le\le u^n$. This latter inequality amounts to
			\[
				x u + u^n \le x + u^{n + 1}.
			\]
			We consider three subcases.
			\begin{itemize}
				\item Suppose $x \sim 1$.

					We then choose $k \in \N$ with $u^{-k} \le x \le u^k$ and apply \Cref{nonarctic_bound_concrete}, which gives the middle inequality in
					\[
						x u + u^n \le u^{k+1} + u^n \le u^{-k} + u^{n+1} \le x + u^{n+1}
					\]
					for sufficiently large $n$.

				\item Suppose $x = \ell \in \N$.

					Then we need to find $n \in \N$ such that
					\[
						\ell u + u^n \le \ell + u^{n + 1}.
					\]
					Multiplying both sides by $u^{-1}$ shows that this is again covered by \Cref{nonarctic_bound_concrete}.
				\item Now for general $x$, let $y \in F$ be a quasi-complement for $x$ in $F/\!\sim$, so that $x + y \sim \ell \in \Nplus$. But what we have already shown is therefore that both $\ell$ and $\ell^{-1}(x + y)$ are upper bounded with respect to $\le\le$ by some power of $u$, say $u^n$. Hence
					\[
						x \le\le x + y = \ell \cdot \ell^{-1} (x + y) \le\le u^{2n},
					\]
					as was to be shown.
			\end{itemize}

			Second, suppose that $\le$ is arctic.
			We then need to show that for every $x \in F^\times$, there is $n \in \N$ such that
			\[
				x u + 2^n u^n \le x + 2^n u^{n + 1},
			\]
			and we do so using the same case distinctions as above.
			\begin{itemize}
				\item Suppose $x \sim 1$.

					We take $n = 1$ and apply \Cref{arctic_main} in order to obtain
					\[
						x u + 2 u \le x u^2 + 2 u \approx x u^2 + 1 + u^2 \approx x + 2 u^2.
					\]
				\item Suppose $x = \ell \in \N$.

					Now choose $n$ such that $\ell \le 2^n$ in $\N$. Then again using \Cref{arctic_main},
					\begin{align*}
						\ell u + 2^n u^n	& = \ell (u + u^n) + (2^n - \ell) u^n \\
									& \approx \ell (1 + u^{n+1}) + (2^n - \ell) u^n \\
									& \le \ell + 2^n u^{n+1},
					\end{align*}
					as was to be shown.
				\item The case of general $x$ reduces to the two previous ones just as above.
			\end{itemize}
		\item What we need to prove is that $4 u + u^{-1} > > 4$, which unfolds to
			\[
				4 u + (4 u + u^{-1}) < 4 + (4 u + u^{-1}) u.
			\]
			This holds non-strictly because of
			\[
				8 u^2 + 1 \le 4 u^3 + 4 u + 1 \le 4 u^3 + 5 u
			\]
			and dividing by $u$. In the arctic case and the max-temperate one, the second inequality is strict by \Cref{nottrop_convex}, which implies the claim.
			Strict inequality holds also in the max-tropical case, since then
			\[
				8 u^2 + 1 \approx 9 u^2 < 9 u^3 = 4 u^3 + 5 u.
			\]

		\item We need to prove $4 u + u^{-1} < < 8 u$, or equivalently
			\[
				4 u^3 + 8 u^2 + u < 8 u^3 + 4 u^2 + 1,
			\]
			which is indeed true by \Cref{arctic_main} and $4 \cdot 3 + 8 \cdot 2 + 1 < 8 \cdot 3 + 4 \cdot 2$.
			\qedhere
	\end{enumerate}
\end{proof}

\begin{prop}
	\label{real_trop}
	Let $F$ be a multiplicatively Archimedean fully preordered semifield with $u > 1$ such that:
	\begin{itemize}
		\item $F$ is of max-temperate or max-tropical type.
		\item $F / \!\sim$ has quasi-complements.
	\end{itemize}
	Then there is a homomorphism $\phi : F \to \mathbb{K}$ with $\mathbb{K} \in \{\R_+, \TR_+\}$ such that for all $x \sim y$ in $F$, we have
	\beq
		\label{phiembed}
		x \le y \quad \Longleftrightarrow \quad \phi(x) \le \phi(y).
	\eeq
\end{prop}

\begin{proof}
	By \Cref{ambient_lemma}, the ambient preorder $\le\le$ turns $F$ into a totally preordered semifield with power universal element $v > > 1$.
	Therefore \Cref{arch_trunc} produces a $\le\le$-monotone homomorphism $\phi : F \to \mathbb{K}$ for $\mathbb{K} \in \{\R_+, \TR_+\}$, where the other two cases are excluded by $v + v^{-1} > > 2$.

	The desired equivalence is again trivial when $x = 0$ or $y = 0$, so we assume $x, y \in F^\times$ and put $y = 1$ without loss of generality.
	Then if $x > 1$, we also obtain that $x > > 1$ is power universal by \Cref{ambient_lemma}, resulting in $\phi(x) > 1$ by \Cref{arch_trunc}.
	Similarly $x < 1$ implies $\phi(x) < 1$.
	Finally, $x \approx 1$ yields $x \approx\approx 1$, and therefore $\phi(x) = 1$.
\end{proof}

We next aim at an analogous statement for the arctic case.
This is formulated in terms of $\R_{(+)}[X] / (X^2)$, the preordered semifield of arctic type introduced in \Cref{dual_numbers_semifield}.
It plays a similarly paradigmatic role as $\R_+$ does in the max-temperate case and $\TR_+$ in the max-tropical case.

\begin{prop}
	\label{arctic_deriv}
	Let $F$ be a multiplicatively Archimedean fully preordered semifield with a power universal element $u > 1$ such that:
	\begin{itemize}
		\item $F$ is of arctic type.
		\item $F / \!\sim$ has quasi-complements.
		\item $(F / \!\sim) \otimes \Z$ is a finite product of fields.
	\end{itemize}
	Then there exists a homomorphism $\psi : F \to \R_{(+)}[X] / (X^2)$ such that for all $x \sim y$ in $F$, we have
	\beq
		\label{psiembed}
		x \le y \quad \Longleftrightarrow \quad \psi(x) \le \psi(y).
	\eeq
	Moreover, if $F$ is also a semialgebra, then $\psi$ can be chosen so as to preserve scalar multiplication by $\R_+$.
\end{prop}

As the proof will show, the assumption on $F / \!\sim \mathop{\otimes} \Z$ can in fact be weakened to formal smoothness over $\Q$ (and probably further). 
We nevertheless phrase the statement in terms of the stronger assumption stating that this ring should be a finite product of fields.
This is more elementary and general enough for how we will use the statement later.

\begin{proof}
	The monotone homomorphisms $\psi : F \to \R_{(+)}[X] / (X^2)$ are precisely the maps of the form
	\[
		\psi(a) = \phi(a) + D(a) X
	\]
	for a degenerate homomorphism $\phi : F \to \R_+$ and a monotone additive map $D : F \to \R$ satisfying the Leibniz rule with respect to $\phi$, which is
	\[
		D(ab) = \phi(a) D(b) + D(a) \phi(b)
	\]
	for all $a, b \in F$.
	If $F$ is a semialgebra, then $\phi$ automatically preserves scalar multiplication since the identity map is the only homomorphism $\R_+ \to \R_+$, and therefore $\psi$ preserves scalar multiplication if and only if $D$ is $\R_+$-linear.
	In either case, we will construct such $\psi$ by constructing its components $\phi$ and $D$ in the following.

	By quasi-complements and \Cref{ambient_lemma}, the ambient preorder $\le\le$ has a power universal element $v > > 1$ with $2 < < v + v^{-1} < < 2 v$, so that \Cref{arch_trunc} provides us with a $\le\le$-monotone homomorphism $\phi : F \to \R_+$.
	Consider next the multiplicative group
	\[
		G \coloneqq \{ x \in F^\times \mid x \sim 1 \}.
	\]
	For $x \in G$, let $D(x)$ be the number associated to it by \Cref{rate_lem}. Then $D : G \to \R$ is an order embedding by construction.
	We prove a few auxiliary statements.
	\begin{enumerate}
		\item For $x, y \sim 1$,
			\beq
				\label{preleibniz}
				D(xy) = D(x) + D(y),
			\eeq
			which in particular implies $D(1) = 0$.

			This follows easily from the definition of $D$.
		\item Moreover, we also have
			\beq
				\label{Dconvex}
				D \left( r x + (1 - r) y \right) = r D(x) + (1 - r) D(y)
			\eeq
			for all rational $r \in [0, 1]$, and for all $r \in [0, 1]$ if $F$ is a semialgebra.

			To see this for rational $r$, we assume $x > 1$ and $y \approx 1$ without loss of generality.
			The claim then follows by an application of \Cref{arctic_formula}.
			In the semialgebra case, the claim for general $r$ follows by monotonicity of $r \mapsto r x + (1 - r) y$ and rational approximation.
		\item\label{lele_to_le} The map
			\[
				F^\times \longrightarrow F, \qquad a \longmapsto \frac{a u + 1}{a + 1}
			\]
			is $\le\le$-to-$\le$-monotone.

			Indeed $a \le\le b$ means exactly that $a u + b \le a + b u$, and therefore the claim follows by
			\begin{align*}
				(a u + 1) (b + 1)	& = a b u + a u + b + 1 \\
							& \le a b u + b u + a + 1 \\
							& = (b u + 1) (a + 1)
			\end{align*}
			and dividing.
		\item For all $a, b \in F^\times$ and $x, y \sim 1$, we have
			\beq
				\label{Dgen_convex}
				D \left( \frac{a x + b y}{a + b} \right) = \frac{ \phi(a) D(x) + \phi(b) D(y) }{ \phi(a) + \phi(b) }.
			\eeq
	
			For the proof, we put $b = y = 1$ without loss of generality. We then show the desired equation
			\[
				D \left( \frac{a x + 1}{a + 1} \right) = \frac{\phi(a)}{\phi(a) + 1} D(x)
			\]
			first for $x = u$.
			For $a \in \Qplus$ it follows by \eqref{Dconvex}. For general $a$, we can use rational approximation in the ambient preorder $\le\le$, which implies the claim by \ref{lele_to_le}.
			Therefore the desired equation holds with $a = u$ for all $r$.

			We now argue that for any $n \in \Nplus$, the desired equation~\eqref{Dgen_convex} holds for $x \sim 1$ if and only if it holds for $x^n$ in place of $x$.
			Using $x \approx \frac{x^n + (n-1)}{n}$, we obtain
			\begin{align*}
				D \left( \frac{a x + 1}{a + 1} \right)	& = D \left( \frac{a \cdot \frac{x^n + (n-1)}{n} + 1}{a + 1} \right) \\
									& = D \left( \frac{1}{n} \cdot \frac{a x^n + 1}{a + 1} + \frac{n-1}{n} \cdot \frac{a + 1}{a + 1} \right) \\
									& = \frac{1}{n} D \left( \frac{a x^n + 1}{a + 1} \right)
			\end{align*}
			by \eqref{Dconvex} and $D(1) = 0$.
			This implies the claim since the right-hand side of~\eqref{Dgen_convex} receives the same factor of $\frac{1}{n}$ from $D(x) = \frac{1}{n} D(x^n)$.

			Finally, another approximation argument shows that the equation therefore holds for all $x \sim 1$, based on the facts established in the previous paragraphs together with monotonicity in $x$.
	\end{enumerate}
	
	We now turn to a number of considerations involving rings.
	By assumption we have a ring isomorphism $(F / \!\sim) \otimes \Z \cong \prod_{i=1}^n K_i$, where the $K_i$ are fields.
	Then every $K_i$ is of characteristic zero, since the image of $F$ in $K_i$ is a subsemifield that is a quotient semifield of $F$, and every quotient of a strict semifield is again a strict semifield (or the zero ring, which is covered by our assumptions in case that the number of factors is $n = 0$).

	The set
	\[
		J \coloneqq \{ b - a \mid a \sim b \textrm{ in } F \}
	\]
	is an ideal in $R \coloneqq (F / \!\approx) \otimes \Z$, namely precisely the kernel of the canonical projection homomorphism $(F / \!\approx) \otimes \Z \longrightarrow (F / \!\sim) \otimes \Z$. We have $J^2 = 0$, since $a \sim b$ and $c \sim d$ imply $(b - a)(d - c) = 0$ in $R$ via \Cref{arctic_main} and
	\[
		bd + ac = ac ( b a^{-1} d c^{-1} + 1 ) \approx ac ( b a^{-1} + d c^{-1} ) = bc + ad,
	\]
	assuming $a,c \neq 0$ without loss of generality. Hence $R$ is a square-zero extension of the quotient ring $R / J \cong \prod_{i=1}^n K_i$. Each $K_i$ is formally smooth over $\Q$~\cite[Corollary~9.3.7]{weibel}, and therefore also their product is formally smooth over $\Q$. (A finite product of formally smooth algebras is formally smooth by lifting of idempotents.)
	In particular, the square-zero extension $R \to R / J$ is split by a ring homomorphism $\beta : R \to R / J$.
	In the semialgebra case, $R$ is an $\R$-algebra in a canonical way, and in this case we have formal smoothness over $\R$ for the same reason, so that we can choose $\beta$ to be $\R$-linear.

	The universal property of $F / \!\approx \mathop{\otimes} \Z$ implies that $\phi : F \to \R_+$ uniquely extends to a ring homomorphism $\hat{\phi} : R \to \R$.
	In the following, we will construct a $\Q$-linear map
	\[
		\hat{D} : R \longrightarrow \R
	\]
	which similarly extends the $D$ defined above, in two stages.
	\begin{enumerate}[resume]
		\item On the nilpotent ideal $J$, taking
			\[
				\hat{D}(b - a) \coloneqq \phi(a) D \left( \frac{b}{a} \right) 
			\]
			for $a \sim b$ in $F^\times$ produces a well-defined map $\hat{D} : J \to \R$.

			Indeed, for well-definedness it is enough to show that adding some $c \in F^\times$ to both terms leaves the right-hand side invariant.
			This will be the special case obtained by taking $d = c$ in the additivity proof of the next item.
		\item\label{hat_preadd} The map $\hat{D} : J \to \R$ is additive. If $F$ is a semialgebra, then it is $\R$-linear.

			Indeed taking $a,b,c,d \in F^\times$ with $a \sim b$ and $c \sim d$, we obtain
			\begin{align*}
				\hat{D} \left( (b - a) + (d - c) \right)	& = \phi(a + c) D \left( \frac{b + d}{a + c} \right) \\
										& = \left( \phi(a) + \phi(c) \right) \, D \left( \frac{a}{a + c} \cdot \frac{b}{a} + \frac{c}{a + c} \cdot \frac{d}{c} \right) \\
										& = \phi(a) D \left( \frac{b}{a} \right) + \phi(c) D \left( \frac{d}{c} \right) \\
										& = \hat{D}(b - a) + \hat{D}(d - c),
			\end{align*}
			where the third step uses \eqref{Dgen_convex}. In the semialgebra case, it is enough to verify preservation of scalar multiplication by positive scalars, which holds since $\phi$ is necessarily a semialgebra homomorphism (because the identity map is the only semiring homomorphism $\R_+ \to \R_+$).
		\item\label{hat_preleibniz} The map $\hat{D} : J \to \R$ satisfies $\hat{D}(ra) = \hat{\phi}(r) \hat{D}(a)$ for all $a \in J$ and $r \in R$. 

			Indeed, writing $a = c - b$ for $b \sim c$ in $F^\times$ and plugging in the definition of $\hat{D}$ shows that this holds for all $r \in F^\times$. But this is nough by linearity in $r$ since $R = F^\times - F^\times$.
		\item Extending by
			\[
				\hat{D}(a) \coloneqq \hat{D}(a - \beta(a))
			\]
			where $\beta : R / J \to R$ is the splitting obtained above, defines an additive map $\hat{D} : R \to \R$ satisfying the Leibniz rule with respect to $\hat{\phi}$.
			If $F$ is a semialgebra, then it is $\R$-linear.

			Since $a - \beta(a) \in J$ for all $a \in R$, this element indeed lies in the domain of $\hat{D}$.
			And since $\beta(a) = 0$ for $a \in J$, it recovers the $\hat{D} : J \to \R$ already defined above, and in particular there is no ambiguity in notation.
			The additivity follows by the additivity of $\beta$ and \ref{hat_preadd}.
			For the Leibniz rule, we take $a, b \in R$ and compute
			\begin{align*}
				\hat{D}(ab)	& = \hat{D}( ab - \beta(ab) ) \\
						& = \hat{D}( a(b - \beta(b)) + \beta(b) (a - \beta(a) )) \\
						& = \hat{\phi}(a) \hat{D}( b - \beta(b)) + \hat{\phi}(\beta(b)) \hat{D}(a - \beta(a)) \\
						& = \hat{\phi}(a) \hat{D}(b) + \hat{\phi}(b) \hat{D}(a),
			\end{align*}
			where the second step uses multiplicativity of $\beta$, the third additivity of $D$ as well as \ref{hat_preleibniz}, and the fourth simply the general definition of $\hat{D}$ as well as the fact that $\hat{\phi}$ factors across $R / J$.

			The claimed $\R$-linearity in the semialgebra case holds by the $\R$-linearity in \ref{hat_preadd} and since $\beta$ is $\R$-linear.
	\end{enumerate}
	Overall, we can therefore define the desired $D : F \to \R$ as the composite map
	\[
		F \longrightarrow (F / \approx) \otimes \Z \stackrel{\hat{D}}{\longrightarrow} \R.
	\]
	The properties of $\hat{D}$ proven above imply that this map is indeed a $\phi$-derivation.
	Moreover, the definition of $\hat{D}$ shows that it restricts to our original $D$ as defined on the multiplicative group $G$.
	This implies the desired equivalence \eqref{psiembed} upon taking $y = 1$ without loss of generality.
\end{proof}

\begin{rem}
	\label{deriv_unique}
	It is interesting to ask how unique the derivation $D$ constructed in the proof is.
	Clearly $D$ can be replaced by any positive multiple of itself, but is there more freedom?
	We answer this question now.

	In terms of the data from the proof, the factored homomorphism $\hat{\phi} : R / J \to \R$ makes $\R$ into an $R/J$-module.
	If $\Delta : R / J \longrightarrow \R$ is now any derivation (over $\Z$), then we can take any derivation $D$ as in the proof and modify it via
	\[
		D' \coloneqq D + \Delta,
	\]
	obtaining another derivation $D'$ that works just as well.
	In particular, $D'$ is still monotone, and in fact $x \sim y$ implies
	\[
		D'(y) - D'(x) = D(y) - D(x).
	\]
	Conversely, if $D$ and $D'$ are two $\phi$-derivations that take the same values on the multiplicative group $G = \{ a \in F^\times \mid a \sim 1 \}$, then the arguments given in the proof show that $D' = D + \Delta$ as above.

	Therefore we can say that $D$, when normalized to $D(u) = 1$, is unique up to the $R / J$-module of derivations $\mathrm{Der}_\Z(R / J, \R)$. Using $R / J \cong \prod_{i=1}^n K_i$ lets us identify this module with the real vector space
	\beq
		\label{ambiggroup}
		\bigoplus_{i=1}^n \mathrm{Der}_\Q(K_i, \R).
	\eeq
	Thus there can be a large ambiguity in the construction of $D$, for example already if $F / \!\sim{} \cong \R_+$.

	If $F$ is a semialgebra, then this ambiguity is attenuated by the additional $\R_+$-linearity condition.
	By the same arguments, the derivation is then unique up to elements of $\bigoplus_{i=1}^n \mathrm{Der}_\R(K_i, \R)$. 
	For example if $F$ is a semialgebra with $F / \!\sim{} \cong \R_+$, then $D$ is unique.
\end{rem}

\section{A stronger catalytic Vergleichsstellensatz}
\label{n2}

\Cref{Imain}, as developed in Part I, concludes both an ``asymptotic'' ordering of the form
\[
	u^k x^n \le u^k y^n \qquad \forall n \gg 1	
\]
and a ``catalytic'' ordering of the form
\[
	a x \le a y
\]
for some nonzero $a \in S$ from the assumption that $\phi(x) < \phi(y)$ for all $\phi \in \Sper{S}$.
Although this is a useful and quite broadly applicable result, the relevant assumption $1 > 0$ makes this result not strong enough for the applications that have been mentioned in the introduction.
The goal of this section is to prove a deeper Vergleichsstellensatz that applies more generally, obtained by putting together the auxiliary results developed in the previous sections.
Here it is.

\begin{thm}[Restatement of \cref{intro_main1}]
	\label{main1}
	Let $S$ be a zerosumfree preordered semidomain with a power universal pair $u_-, u_+ \in S$ and such that:
	\begin{itemize}
		\item $S / \!\sim$ has quasi-complements and quasi-inverses.
		\item $\Frac(S / \!\sim) \otimes \Z$ is a finite product of fields.
	\end{itemize}
	Let nonzero $x, y \in S$ with $x \sim y$ satisfy the following:
	\begin{itemize}
		\item For every nondegenerate monotone homomorphism $\phi : S \to \mathbb{K}$ with trivial kernel and $\mathbb{K} \in \{\R_+, \R_+^\op, \TR_+, \TR_+^\op\}$,
			\[
				\phi(x) < \phi(y).
			\]
		\item For every monotone additive map $D : S \to \R$, which is a $\phi$-derivation for some degenerate homomorphism $\phi : S \to \R_+$ with trivial kernel and satisfies $D(u_+) = D(u_-) + 1$,
			\[
				D(x) < D(y).
			\]
	\end{itemize}
	Then there is nonzero $a \in S$ such that $a x \le a y$. 

	Moreover, if $S$ is also a semialgebra, then it is enough to consider $\R_+$-linear derivations $D$ in the assumptions.
\end{thm}

Of course, if nonzero $a$ with $a x \le a y$ exists, then this conversely implies the non-strict inequalities $\phi(x) \le \phi(y)$ and $D(x) \le D(y)$ for all $\phi$ and $D$ as in the statement.

\begin{proof}
	As sketched in the proof of \Cref{arctic_deriv}, the monotone $\phi$-derivations $D$ for degenerate $\phi : S \to \R_+$ are in canonical bijection with the monotone homomorphisms $S \to \R_{(+)}[X] / (X^2)$ whose composition with the projection $\R_{(+)}[X] / (X^2) \to \R_+$ coincides with $\phi$.

	We start the proof with the case where $S$ is a preordered semifield $F$ of polynomial growth.
	In this case, the two additional hypotheses amount to $F$ having quasi-complements and that $F / \!\sim \mathop{\otimes} \Z$ is a finite product of fields.
	Also $u \coloneqq u_+ u_-^{-1}$ is a power universal element in $F$.
	We moreover assume $y = 1$ without loss of generality.
	Considering this case will take the bulk of the proof; we will generalize from there in the final two paragraphs.

	The goal is to prove $x \le 1$.
	If $x \not\le 1$ held, then by \Cref{semifield_main} we could find a total semifield preorder $\preceq$ which extends $\le$ and satisfies $x \succ 1$.
	We fix such $\preceq$ from now on.
	Now consider the layer preorder $\preceq_u$, as defined in \Cref{layer_preorder}.
	As we have seen, it turns $F$ into a multiplicatively Archimedean fully preordered semifield.
	Moreover, $\preceq_u$ still extends $\le$, for the following reason.
	If $a \le 1$ is arbitrary, then the first condition in the definition of $a \preceq_u 1$ according to~\Cref{layer_preorder} holds by the power universality of $u$ for $\le$, and the second one holds trivially by $a \preceq 1 \preceq u$. Hence $a \preceq_u 1$, so that $\le$ is indeed contained in $\preceq_u$.
	We also note $u \succ_u 1$ as well as the important inequality
	\beq
		\label{x_amb}
		x \succeq_u 1
	\eeq
	for future use, which follows from the definition of $\preceq_u$ using $u^{-k} \le x \le u^k$ for some $k$ and $x \succeq 1$.
	We now distinguish two cases.
	\begin{enumerate}[label=(\arabic*)]
		\item $\preceq_u$ is tropical or temperate.
	\end{enumerate}
	By reversing all preorders and replacing $x$ by $x^{-1}$ if necessary, we can assume without loss of generality that $\preceq_u$ is max-tropical or max-temperate.
	Applying \Cref{real_trop} to $(F,\preceq_u)$ produces a $\preceq_u$-monotone homomorphism $\phi : F \to \mathbb{K}$ for $\mathbb{K} \in \{\R_+, \TR_+\}$ with $\phi(u) > 1$. The inequality~\eqref{x_amb} produces $\phi(x) \ge 1$.
	On the other hand, since $\preceq_u$ extends $\le$, it is clear that $\phi$ is also $\le$-monotone, and $u \ge 1$ together with $\phi(u) > 1$ imply that $\phi$ is nondegenerate.
	Therefore the assumption applies and gives $\phi(x) < 1$, a contradiction.
	\begin{enumerate}[label=(\arabic*),resume]
		\item $\preceq_u$ is arctic. 
	\end{enumerate}
	In this case, we apply \Cref{arctic_deriv} to $(F,\preceq_u)$. Writing $\simeq$ for the equivalence relation generated by $\preceq_u$, we need to verify that the semifield $F / \!\simeq$ has quasi-complements and that $F / \!\simeq \mathop{\otimes} \Z$ is a finite product of fields. But these are both true since $F / \!\simeq$ is a quotient of $F / \!\sim$ (because $\preceq_u$ extends $\le$) and these statements descend to quotients.
	We therefore obtain a degenerate homomorphism $\phi : F \to \R_+$ and a $\preceq_u$-monotone $\phi$-derivation $D : F \to \R$ with $D(u) > 0$.
	These properties in particular imply $D(x) \ge 0$ by~\eqref{x_amb}.
	By rescaling, we can assume $D(u) = 1$ without loss of generality.
	But again since $\preceq_u$ extends $\le$, we know that $D$ is in particular $\le$-monotone, so that the assumption applies.
	This results in $D(x) < 0$, a contradiction.

	This proves the desired statement in the semifield case.
	It remains to reduce the general case to this one.
	We thus verify that the preordered semifield $F \coloneqq \Frac(S)$ satisfies the relevant assumptions. Clearly since $\Frac(S) / \!\sim\: = \Frac(S / \!\sim)$ by \cref{frac_sim}, the existence of quasi-complements in $F / \!\sim$ follows from the existence of quasi-complements and quasi-inverses in $S / \!\sim$ by \Cref{frac_quasi}.
	Similarly, $F / \!\sim \mathop{\otimes} \Z$ is also a finite product of fields.

	Applying the statement to $F$ then produces the desired result as follows.
	The relevant inequalities $\phi\left(\frac{x}{1}\right) < \phi\left(\frac{y}{1}\right)$ and $D\left(\frac{x}{1}\right) < D\left(\frac{y}{1}\right)$ hold because every such map restricts along the homomorphism $S \to \Frac(S)$ to a map of the corresponding type on $S$, where the assumed inequalities apply. Since a homomorphism $\phi : \Frac(S) \to \mathbb{K}$ necessarily has trivial kernel, therefore so does its restriction to $S$.
	If $D : \Frac(S) \to \R$ is a $\phi$-derivation at a degenerate homomorphism $\phi : \Frac(S) \to \R_+$, then the restriction $D|_S$ is clearly also a $\phi|_S$-derivation, and $\phi|_S$ is degenerate as well since $\phi$ is degenerate.
	We conclude that $\frac{x}{1} \le \frac{y}{1}$ in $\Frac(S)$.
	By the definition of the preorder on $\Frac(S)$ of~\eqref{frac_preorder}, this means that there is some nonzero $a \in S$ with $a x \le a y$, as desired.
\end{proof}

\begin{rem}
	In order to determine the ring $\Frac(S / \!\sim) \otimes \Z$ in practice, it is worth noting that the three types of constructions involved in its definition commute: we can perform the quotient by $\sim$, the localization at nonzero elements of $S$ and the Grothendieck construction in either order.
	This is most easily seen by the universal property: $\Frac(S / \!\sim) \otimes \Z$ is initial in the category of rings $R$ equipped with a semiring homomorphism $S \to R$ which identifies all $\sim$-related elements and maps all nonzero elements to the group of units $R^\times$.
\end{rem}

\begin{ex}
	\label{free_ex}
	Consider $\N[\underline{X}]$, the free semiring in $d$ variables $\underline{X} = (X_1, \ldots, X_d)$, equipped with the semiring preorder generated by
	\[
		X_1 \ge 1, \quad \ldots, \quad X_d \ge 1.
	\]
	This preordered semiring is of polynomial growth, since $u_- = 1$ and $u_+ = \prod_i X_i$ form a power universal pair.
	We also have $\N[\underline{X}] / \!\sim{} \cong \N$, so that the other assumptions on $S$ in \Cref{main1} obviously hold as well.
	The preorder can be characterized as $f \le g$ if and only if there is a finitely supported family of polynomials $(p_\alpha)_{\alpha \in \N^d}$ such that\footnote{For a multiindex $\alpha$, we use the shorthand notation $\underline{X}^\alpha \coloneqq \prod_{i=1}^d X_i^{\alpha_i}$.}
	\[
		f = \sum_{\alpha \in \N^d} p_\alpha, \qquad g = \sum_{\alpha \in \N^d} p_\alpha \underline{X}^\alpha.
	\]
	This works because this relation is a semiring preorder, and can be seen to be the smallest semiring preorder with $X_i \ge 1$ for all $i$.
	Indeed the only part of this statement that is not straightforward is the transitivity.
	Given $f \le g \le h$, the desired $f \le h$ follows upon choosing a common refinement of the two given decompositions of $g$ by virtue of the Riesz decomposition property\footnote{If $(p_i)_{i \in I}$ and $(q_j)_{j \in J}$ are finitely supported families of polynomials with $\sum_i p_i = \sum_j q_j$, then there is a doubly indexed family $(r_{ij})_{i \in I, j \in J}$ with $p_i = \sum_j r_{ij}$ and $q_j = \sum_j r_{ij}$. One possible proof of this is by induction on the size of the support.}.

	In order to apply \Cref{main1}, we then first classify the monotone homomorphisms $\N[\underline{X}] \to \mathbb{K}$ for $\mathbb{K} \in \{\R_+, \R_+^\op, \TR_+, \TR_+^\op\}$.
	Using the fact that $\N[\underline{X}]$ is the free semiring, it follows that the homomorphisms to $\R_+$ with trivial kernel are precisely the evaluation maps $f \mapsto f(r)$ for $r \in \Rplus^d$. Similarly, the homomorphisms $\N[\underline{X}] \to \TR_+$ are given by optimization over the Newton polytope in a fixed direction $\beta \in \R^d$,\footnote{See I.2.9 for more detail.}
	\[
		\N[\underline{X}] \longrightarrow \TR_+, \qquad f \longmapsto \max_{s \,\in\, \Newton(f)} \langle \beta, s \rangle.
	\]
	By checking when the generating preorder relations $X_i \ge 1$ are preserved, we therefore obtain the following classification of the nondegenerate monotone homomorphisms $\phi : \N[\underline{X}] \to \mathbb{K}$ with trivial kernel:
	\begin{itemize}
		\item For $\mathbb{K} = \R_+$, the evaluation maps $f \mapsto f(r)$ with $r \in \Rplus^d$ having components $r_i \ge 1$, not all $1$.
		\item For $\mathbb{K} = \R_+^\op$, the evaluation maps $f \mapsto f(r)$ with $r \in \Rplus^d$ having components $r_i \le 1$, not all $1$.
		\item For $\mathbb{K} = \TR_+$, the optimization maps $f \mapsto \max_{s \,\in\, \Newton(f)} \langle \beta, s \rangle$ for $\beta \in \R^d$ having components $\beta_i \ge 0$ not all $0$.
		\item For $\mathbb{K} = \TR_+^\op$, the optimization maps $f \mapsto \max_{s \,\in\, \Newton(f)} \langle \beta, s \rangle$ for $\beta \in \R^d$ having components $\beta_i \le 0$ not all $0$.
	\end{itemize}
	Concerning the derivations, there is exactly one degenerate homomorphism $\phi : \N[\underline{X}] \to \R_+$, namely the evaluation map $f \mapsto f(\underline{1})$. The monotone $\phi$-derivations $\N[\underline{X}] \to \R$ are parametrized by $\gamma \in \R_+^d$ and given by
	\begin{itemize}
		\item $f \mapsto \langle \gamma, \nabla f \rangle$, where $\nabla f \in \R_+^d[\underline{X}]$ is the gradient of $f$.
	\end{itemize}
	The normalization condition $D(u_+) = D(u_-) + 1$ amounts to the constraint $\sum_i \gamma_i = 1$. Since every such $\gamma$ is a convex combination of the standard basis vectors, it is sufficient to impose the assumed inequalities on these.
	
	Hence \Cref{main1} instantiates to the following result. Suppose that $f, g \in \N[\underline{X}]$ satisfy the following conditions:
	\begin{itemize}
		\item $f(\underline{1}) = g(\underline{1})$.
		\item $f(r) < g(r)$ for all $r \in [1,\infty)^d \setminus \{\underline{1}\}$.
		\item $f(r) > g(r)$ for all $r \in (0,1]^d \setminus \{\underline{1}\}$.
		\item For every $\beta \in \R_+^d \setminus \{0\}$,
			\begin{align*}
				\max_{s \,\in\, \Newton(f)} \langle \beta, s \rangle & < \max_{s \,\in\, \Newton(g)} \langle \beta, s \rangle,	\\[5pt]
				\min_{s \,\in\, \Newton(f)} \langle \beta, s \rangle & < \min_{s \,\in\, \Newton(g)} \langle \beta, s \rangle.
			\end{align*}
		\item For all $i = 1, \ldots, d$,
			\[
				\frac{\partial f}{\partial X_i} (\underline{1}) < \frac{\partial g}{\partial X_i} (\underline{1})
			\]
	\end{itemize}
	Then there is a nonzero polynomial $h \in \N[\underline{X}]$ and a family $(p_\alpha)_{\alpha \in \N^d}$ of polynomials $p_\alpha \in \N[\underline{X}]$ such that
	\[
		h f = \sum_{\alpha \in \N^d} p_\alpha, \qquad h g = \sum_{\alpha \in \N^d} p_\alpha \underline{X}^\alpha.
	\]
	Conversely, this property is clearly sufficient to imply the above conditions with non-strict inequality.

	Using the fact that the identity map is the only semiring homomorphism $\R_+ \to \R_+$, it is straightforward to see that the same result holds with $\R_+$ coefficients instead of $\N$ coefficients.
\end{ex}

\section{A stronger asymptotic Vergleichsstellensatz}
\label{n3}

Throughout this section, we also assume that $S$ is a zerosumfree semidomain, and of polynomial growth with respect to a power universal element $u \in S$.

\Cref{Imain} also makes statements about ``asymptotic'' ordering relations of the form $u^k x^n \le u^k y^n$ for all $n \gg 1$. 
Its proof conducted in Part I was based on a suitable definition of \emph{test spectrum} and a proof that the test spectrum is a compact Hausdorff space.
We do not yet have a sufficiently general definition of test spectrum to achieve a similar feat at the level of generality of \Cref{main1}, but we do under the stronger assumption that
\beq
	\label{Ssim_Rd}
	S / \!\sim{}\! \cong \R_{>0}^d \cup \{0\}
\eeq
for some $d \in \N$.\footnote{Note that this amounts to $S / \!\sim{}\! \cong \R_+$ in the case $d = 1$, and we interpret it as $S / \!\sim{}\! \cong \B$ for $d = 0$.}
If this is the case, then $S / \!\sim$ obviously has quasi-complements and quasi-inverses and is such that $\Frac(S / \!\sim) \otimes \Z \cong \R^d$ is a finite product of fields, so that the algebraic assumptions on $S / \!\sim$ in \Cref{main1} are clearly satisfied.

The assumption~\eqref{Ssim_Rd} is most conveniently formulated a little differently: we assume that $S$ is a preordered semiring equipped with a fixed surjective\footnote{The surjectivity requirement can arguably be relaxed and replaced by a suitable arithmetical conditions in the image of $\|\cdot\|$ in $\R_{>0}^d \cup \{0\}$, but we have not worked this out in detail since the surjectivity holds in our applications~\cite{major,arw}.} homomorphism
\[
	\|\cdot\| : S \longrightarrow \R_{>0}^d \cup \{0\}
\]
with trivial kernel and such that
\[
	x \le y \quad \Longrightarrow \quad \|x\| = \|y\| \quad \Longrightarrow \quad x \sim y.
\]
Under these assumptions, we obtain the desired~\eqref{Ssim_Rd}, where the isomorphism is implemented by $\|\cdot\|$ itself. 
The components of $\|\cdot\|$ are degenerate homomorphisms
\begin{equation}
	\label{norm_comps}
	\|\cdot\|_1, \: \ldots, \: \|\cdot\|_d : S \longrightarrow \R_+.
\end{equation}

\begin{rem}
	\begin{enumerate}
		\item The homomorphisms~\eqref{norm_comps} are the only degenerate homomorphisms $S \to \R_+$.\footnote{This would not necessarily be the case without the surjectivity requirement on $\|\cdot\|$. For example if $d = 1$, and if the image of $\|\cdot\|$ is a semiring isomorphic to $\N[X]$, then this fails in a particularly bad way due to the abundance of homomorphisms $\N[X] \to \R_+$.}
			To see this for $d \ge 1$, use~\eqref{Ssim_Rd} and note that every homomorphism $\R_{>0}^d \cup \{0\} \rightarrow \R_+$ extends uniquely to a ring homomorphism $\R^d \to \R$, and the only such homomorphisms are the projections. For $d = 0$, the statement is that there is no homomorphism $\B \to \R_+$ at all, which is clear.
		\item The notation $\|\cdot\|$ is modelled after norms in functional analysis, since these sometimes have homomorphism properties when restricted to a semiring of suitably ``positive'' elements.
		\item For example if $G$ is an abelian group and $\R_+[G]$ the associated group semialgebra, then the $\ell^1$-norm is a semiring homomorphism $\R_+[G] \to \R_+$, making it an instance of the above with $d = 1$.
	\end{enumerate}
\end{rem}

\subsection{The test spectrum}

We now move towards the relevant notion of spectrum, using the multiplicative picture of $\TR_+$.
We need a little more preparation to deal with the derivations, and in particular with the ambiguity discussed in \Cref{deriv_unique}.

\begin{defn}
	For $i = 1,\ldots,d$, two $\|\cdot\|_i$-derivations $D, D' : S \to \R$ are \newterm{interchangeable} if their difference factors through $S / \!\sim$.
\end{defn}

In other words, interchangeability of derivations $D$ and $D'$ means that there is a $\Q$-linear derivation $\Delta : \R \to \R$ such that\footnote{To see this, note that every $\|\cdot\|_i$-derivation uniquely extends to an additive map $\R^d \to \R$ that is a derivation with respect to the $i$-projection map $\R^d \to \R$. But every such derivation itself factors through the $i$-th projection map, and hence our $D' - D$ actually factors through $\|\cdot\|_i$.}
\[
	D'(x) = D(x) + \Delta(\|x\|_i) \qquad \forall x \in S.
\]
The real vector space of such $\Delta$'s is exactly $\mathrm{Der}_\Q(\R,\R)$, which is infinite-dimensional\footnote{This is by the standard fact that the module of K\"ahler differentials for a transcendental field extension in characteristic zero is free with basis given by a transcendence basis of the extension~\cite[Theorem~16.14]{eisenbud}.}.
Since two interchangeable derivations satisfy
\[
	D'(y) - D'(x) = D(y) - D(x)
\]
whenever $x \sim y$, it is sufficient to consider only one of these derivations in our Vergleichsstellens\"atze in the spectral preordering.
In the definition of the test spectrum below, this is why we only consider interchangeability classes of derivations.
If $S$ is a semialgebra, then every interchangeability class has a canonical representative given by an $\R_+$-linear derivation, which is a very convenient feature worth keeping in mind.

\begin{defn}
	\label{test_spectrum}
	The \newterm{test spectrum} of $S$ is the disjoint union
	\begin{align*}
		\Sper{S} \: \coloneqq \:	& \: \{\text{\normalfont{monotone homs }} S \to \R_+ \text{\normalfont{ or }} S \to \R_+^\op \text{\normalfont{ with trivial kernel}} \} \setminus \{\|\cdot\|_1,\ldots,\|\cdot\|_d\} \\[2pt]
					& \sqcup \{\text{\normalfont{monotone homs }} S \to \TR_+ \text{\normalfont{ with }} \phi(u) = e \text{\normalfont{ and trivial kernel}} \} \\[2pt]
					& \sqcup \{\text{\normalfont{monotone homs }} S \to \TR_+^\op \text{\normalfont{ with }} \phi(u) = e^{-1} \text{\normalfont{ and trivial kernel}} \} \\[-2pt]
				& \sqcup \bigsqcup_{i=1}^d \{\text{\normalfont{monotone }} \|\cdot\|_i\text{\normalfont{-derivations }} S \to \R \text{\normalfont{ with }} D(u) = 1 \text{\normalfont{ modulo interchangeability}} \}.
	\end{align*}
\end{defn}

There are thus five types of points of $\Sper{S}$, and these five types match the five types of \Cref{five_types_def}.

As in the simpler case treated in Part I, our goal is to turn $\Sper{S}$ into a compact Hausdorff space.
The relevant topology will again be the weak topology with respect to a certain class of maps. In the present case,
these are the \newterm{logarithmic comparison maps}, defined for nonzero $x, y \in S$ with $x \sim y$ as
\beq
	\label{lev_defn}
	\lc_{x,y}(\phi) \coloneqq \frac{\log \frac{\phi(y)}{\phi(x)}}{\log \phi(u)}, \quad\qquad \lc_{x,y}(D) \coloneqq \frac{D(y) - D(x)}{\|x\|_i},
\eeq
where the first equation applies in all four non-derivation cases and the second equation in the $\|\cdot\|_i$-derivation case for $i = 1,\ldots,d$, respectively, where one may want to keep in mind that $\|x\|_i = \|y\|_i$.

A few further clarifying comments on this definition are in order:

\begin{itemize}
	\item The definition in the derivation case clearly respects interchangeability, making $\lc_{x,y}$ well-defined on $\Sper{S}$.
	\item The denominator in the definition of $\lc_{x,y}(\phi)$ does not vanish for any $\phi$, since $\phi(u) = 1$ would imply by power universality that $\phi$ is degenerate, and therefore equal to one of the $\|\cdot\|_i$, which we have assumed not to be the case.
		In fact, the denominator is positive by $\phi(u) > 1$ for $\phi : S \to \R_+$ and $\phi : S \to \TR_+$, and it is likewise negative for $\phi : S \to \R_+^\op$ and $\phi : S \to \TR_+^\op$.
		
	\item It follows that $\lc_{x,y} \ge 0$ if $x \le y$.
	\item The denominator also results in the convenient normalization $\lc_{u, 1} = 1$.
	\item The denominator $\|x\|$ in the definition of $\lc_{x,y}(D)$ is what makes the equation
		\beq
			\label{lev_rescale}
			\lc_{ax,ay} = \lc_{x,y}
		\eeq
		hold on all of $\Sper{S}$ for all nonzero $a \in S$.
	\item We also have the following cocycle equation: for nonzero $x \sim y \sim z$,
		\[
			\lc_{x,z} = \lc_{x,y} + \lc_{y,z}.
		\]
\end{itemize}

\begin{defn}
	\label{2test_topology}
	$\Sper{S}$ carries the weakest topology which makes the logarithmic comparison maps
	\[
		\lc_{x, y} \: : \: \Sper{S} \longrightarrow \R
	\]
	continuous for all nonzero $x, y \in S$ with $x \sim y$.
\end{defn}

The following compactness statement is now the analogue of I.7.9, where also the proof is conceptually similar.

\begin{prop}
	\label{chaus2}
	With these definitions, $\Sper{S}$ is a compact Hausdorff space.
\end{prop}

\begin{proof}
	We first note that $\Sper{S} = \Sper{\Frac(S)}$ by~\eqref{lev_rescale} and the fact that interchangeability of derivations on $S$ is equivalent to that on $\Frac(S)$, since $S / \!\sim{}\! \cong \Frac(S) / \!\sim$ by~\eqref{Ssim_Rd}.
	We therefore assume without loss of generality that $S$ is a semifield $F$.
	Again by~\eqref{lev_rescale}, on $\Sper{F}$ we have
	\[
		\lc_{x,y} =  \lc_{1,\frac{y}{x}}.
	\]
	The topology on $\Sper{F}$ is therefore equivalently generated by the \newterm{logarithmic evaluation maps}
	\[
		\lev_x(\phi) \coloneqq \frac{\log \phi(x)}{\log \phi(u)}, \qquad\quad
		\lev_x(D) \coloneqq D(x),
	\]
	and these are parametrized by $x \sim 1$ in $F$.
	These maps satisfy a degenerate form of the Leibniz rule,
	\[
		\lev_{xy} = \lev_x + \lev_y,
	\]
	as well as monotonicity in $x$ and
	\[
		\lev_u = 1, \qquad \lev_1 = 0.
	\]

	We then start the proof of the claim by showing Hausdorffness first.
	Indeed the logarithmic evaluation map
	\[
		\lev_{\frac{u + 1}{2}} \: : \: \Sper{F} \longrightarrow \R
	\]
	nicely distinguishes the types as follows:
	\begin{itemize}
		\item For $\phi : F \to \TR_+$, we have $\lev_{\frac{u+1}{2}}(\phi) = 1$.
		\item For $\phi : F \to \R_+$, we have $\frac{1}{2} < \lev_{\frac{u+1}{2}}(\phi) < 1$, where the former is by $\frac{r + 1}{2} > r^{1/2}$ for all $r > 1$.
		\item For $D : F \to \R$ a $\|\cdot\|_i$-derivation, we have $\lev_{\frac{u+1}{2}}(D) = \frac{1}{2}$.
		\item For $\phi : F \to \R_+^\op$, we have $0 < \lev_{\frac{u+1}{2}}(\phi) < \frac{1}{2}$, where the latter is by $\frac{r + 1}{2} < r^{1/2}$ for all $r \in (0,1)$.
		\item For $\phi : F \to \TR_+^\op$, we have $\lev_{\frac{u+1}{2}}(\phi) = 0$.
	\end{itemize}
	In particular, this shows that any two points of distinct types can be separated, namely by the continuous function $\lev_{\frac{u + 1}{2}} \in C(\Sper{F})$.
	To separate two points of the same type, we consider each one of the five cases separately, and show in each case that if two points cannot be separated, then they are equal:
	\begin{itemize}
		\item For $\phi, \psi : F \to \TR_+$, suppose that $\lev_x(\phi) = \lev_x(\psi)$ for all $x \sim 1$.

			Then $\phi(u) = \psi(u) = e$ implies $\phi(x) = \psi(x)$ for all $x \sim 1$ by definition of the logarithmic evaluation maps. For arbitrary $a \in F^\times$, let $k \in \N$ be such that $e^{-k} \le \phi(a), \psi(a) \le e^k$. Then
			\[
				e^{-k} \phi(a) = \phi\left( \frac{a + u^{-k}}{a + u^k} \right)
				= \psi\left( \frac{a + u^{-k}}{a + u^k} \right) = e^{-k} \psi(a),
			\]
			implying that $\phi = \psi$.
		\item For $\phi, \psi : F \to \R_+$, suppose that $\lev_x(\phi) = \lev_x(\psi)$ for all $x \sim 1$.

			Then using the above $\lev_{\frac{u+1}{2}}$ shows $\phi(u) = \psi(u)$ after a short calculation. Therefore again $\phi(x) = \psi(x)$ for all $x \sim 1$.
			Then for arbitrary $a \in F^\times$, the desired $\phi(a) = \psi(a)$ follows by
			\[
				\frac{\phi(a) + \phi(u)}{\phi(a) + 1} = \phi\left( \frac{a + u}{a + 1} \right)
				= \psi\left( \frac{a + u}{a + 1} \right) = \frac{\psi(a) + \psi(u)}{\psi(a) + 1}
			\]
			and some calculation, using $\phi(u) = \psi(u) \neq 1$ by nondegeneracy.
		\item For $D, D' : S \to \R$, where $D$ is an $\|\cdot\|_i$-derivation and $D'$ is a $\|\cdot\|_j$-derivation for some $i,j \in \{1,\ldots,d\}$, suppose that $\lev_x(D') = \lev_x(D)$ for all $x \sim 1$.

			We first show that this requires $j = i$.
			Indeed every derivation ``remembers'' the degenerate homomorphism $\|\cdot\|_i$ at which it is defined: using additivity and the Leibniz rule together with $\|u\|_i = 1$, we obtain that for every $a \in F$,
			\begin{align*}
				D \left( \frac{ua + 1}{a + u} \right) & = \frac{D(ua) \left( \|a\|_i + 1 \right) - D(a + u) \left( \|a\|_i + 1 \right)}{\left( \|a\|_i + 1 \right)^2} \\
					& = \frac{D(a) + \|a\|_i D(u) - D(a) - D(u)}{\|a\|_i + 1} \\
					& = \frac{\|a\|_i - 1}{\|a\|_i + 1} D(u).
			\end{align*}
			Since this can be solved uniquely for $\|a\|_i$, and the same applies to $D'$ with respect to $\|a\|_j$, the assumptions on $D$ and $D'$ imply that $\|a\|_i = \|a\|_j$ for all $a \in F$, and hence $j = i$.

			But then $D' - D$ is a $\|\cdot\|_i$-derivation that factors across $F / \!\sim{}$. Therefore $D'$ and $D$ are interchangeable and represent the same point of $\Sper{F}$.
			
		\item The remaining two cases involving $\R_+^\op$ and $\TR_+^\op$ work similarly as the first two.
	\end{itemize}
	This completes the proof of Hausdorffness.

	For compactness, we characterize $\Sper{F}$ as a closed subspace of the product space $\prod_{x \sim 1} [-k_x, k_x]$, which is compact by Tychonoff's theorem.
	Here, $k_x \in \N$ is such that the power universality inequalities
	\[
		x \le u^{k_x}, \qquad x u^{k_x} \ge 1
	\]
	hold.
	Given an element $\nu \in \prod_{x \sim 1} [-k_x,k_x]$, we claim that it corresponds under the logarithmic evaluation maps to a point of $\Sper{F}$ if and only if the following conditions hold for all $x, y \sim 1$ and all nonzero $a$:
	\begin{enumerate}
		\item\label{leibniz_nu} $\nu_{xy} = \nu_x + \nu_y$.
		\item $\nu_u = 1$.
		\item If $x \ge 1$, then $\nu_x \ge 0$.
		\item\label{fraction_sign} We have
			\[
				\left( 0 \le \nu_{\frac{x + a}{1 + a}} \le \nu_x \right) \quad \lor \quad 
				\left( \nu_x \le \nu_{\frac{x + a}{1 + a}} \le 0 \right).
			\]
	\end{enumerate}
	The proof is complete once this claim is established, since all of these conditions are clearly closed.

	It is straightforward to verify that these conditions hold for a spectral point, so we focus on the converse.
	For given $\nu$, the task is to extend it to a spectral point $\phi$ or $D$ defined on all of $F$.
	To this end, consider the new preorder relation defined for $x, y \in F^\times$ by
	\[
		x \preceq y \quad : \Longleftrightarrow \quad (x \sim y \:\:\land\:\: \nu_{y x^{-1}} \ge 0).
	\]
	Declaring additionally $0 \preceq 0$, the above properties \ref{leibniz_nu}--\ref{fraction_sign} then imply that $\preceq$ makes $F$ into a preordered semifield, where the monotonicity of addition in particular relies on \ref{fraction_sign}.
	Note that $\preceq$ extends $\le$ and the quotient of $F$ by the equivalence relation generated by $\preceq$ is still $\R_{>0}^d \cup \{0\}$. 
	Moreover, $(F,\preceq)$ is clearly a multiplicatively Archimedean fully preordered semifield with $u \succ 1$.

	We once more distinguish types: 
	\begin{itemize}
		\item If $\nu_{\frac{u + 1}{2}} > \frac{1}{2}$, then we obtain $u^2 + 2 u + 1 \succ 4 u$ from
			\[
				\nu_{\frac{u^2 + 2u + 1}{4}} = 2 \nu_{\frac{u + 1}{2}} > 1 = \nu_u,
			\]
			which implies that $\preceq$ is max-tropical or max-temperate.
			Therefore \Cref{real_trop} applies and produces a homomorphism $\phi : F \to \mathbb{K}$ with $\mathbb{K} \in \{\R_+, \TR_+\}$ such that for all $x \sim y$,
			\[
				x \preceq y \quad \Longleftrightarrow \quad \phi(x) \le \phi(y).
			\]
			In particular, $\phi$ is still $\le$-monotone and nondegenerate.
			For $x \sim 1$, the desired equation $\nu_x = \frac{\log \phi(x)}{\log \phi(u)}$ then holds as a consequence of the definition of $\preceq$ and \Cref{rate_lem} using rational approximation for the real number $\nu_x$.
		\item If $\nu_{\frac{u + 1}{2}} = \frac{1}{2}$, then $\nu$ will correspond to a derivation.
			Indeed $u^2 + 2 u + 1$ and $4 u$ are now $\preceq$-equivalent, and if $u + u^{-1} \succ 2$ or $u + u^{-1} \prec 2$ was the case, then the cancellation criterion of \Cref{cancel2plus} would imply $u^2 + 2 u + 1 \succ 4 u$ or $u^2 + 2 u + 1 \prec 4 u$.
			Hence $\preceq$ is arctic and \Cref{arctic_deriv} applies, and the resulting homomorphism $\psi : F \to \R_{(+)}[X] / (X^2)$ must have components given by $\|\cdot\|_i$ for some $i \in \{1,\ldots,d\}$, since these are the only degenerate homomorphisms $F \to \R_+$, together with some $\|\cdot\|_i$-derivation $D : F \to \R$.
			The claim $\nu_x = D(x)$ for nonzero $x$ now follows as in the previous item, while also using \Cref{arctic_formula} again.
		\item If $\nu_{\frac{u + 1}{2}} < \frac{1}{2}$, we can proceed as in the first case.
			The only difference is that \Cref{real_trop} needs to be applied to $\preceq^\op$ since $\preceq$ is now min-temperate or min-tropical.
			\qedhere
	\end{itemize}
\end{proof}

Here is now our second main result.

\begin{thm}[Restatement of \cref{main2_intro}]
	\label{main2}
	Let $S$ be a preordered semiring with a power universal element $u \in S$.
	Suppose that for some $d \in \N$, there is a surjective homomorphism $\|\cdot\| : S \to \R_{>0}^d \cup \{0\}$ with trivial kernel and such that
	\[
		a \le b \quad \Longrightarrow \quad \|a\| = \|b\| \quad \Longrightarrow \quad a \sim b.
	\]
	Let $x,y \in S$ be nonzero with $\|x\| = \|y\|$.
	Then the following are equivalent:
	\begin{enumerate}
		\item\label{spectral_ineqs}
			\begin{itemize}
				\item For every nondegenerate monotone homomorphism $\phi : S \to \mathbb{K}$ with trivial kernel and $\mathbb{K} \in \{\R_+, \R_+^\op, \TR_+, \TR_+^\op\}$,
					\[
						\phi(x) \le \phi(y).
					\]
				\item For every $i = 1,\ldots,d$ and monotone $\|\cdot\|_i$-derivation\footnote{Recall from~\eqref{norm_comps} that we write $\|\cdot\|_i$ for the $i$-th component of $\|\cdot\|$.} $D : S \to \R$ with $D(u) = 1$,
					\[
						D(x) \le D(y).
					\]
			\end{itemize}
		\item\label{alg_ineqs} For every $\eps > 0$, we have
			\[
				x^n \le u^{\lfloor \eps n \rfloor} y^n	\qquad \forall n \gg 1.
			\]
	\end{enumerate}
	Moreover, suppose that the inequalities in~\ref{spectral_ineqs} are all strict.
	Then also the following hold:
	\begin{enumerate}[resume]
		\item\label{main2_asymp} There is $k \in \N$ such that
			\[
				u^k x^n \le u^k y^n \qquad \forall n \gg 1.
			\]
		\item\label{main2_power} If $y$ is power universal as well, then
			\[
				x^n \le y^n \qquad \forall n \gg 1.
			\]
		\item\label{main2_catal} There is nonzero $a \in S$ such that
			\[
				a x \le a y.
			\]
			Moreover, there is $k \in \N$ such that $a \coloneqq u^k \sum_{j=0}^n x^j y^{n-j}$ for any $n \gg 1$ does the job.
	\end{enumerate}
	Finally, if $S$ is also a semialgebra, then all statements also hold with only $\R_+$-linear derivations $D$ in \ref{spectral_ineqs}.
\end{thm}

Putting $d = 0$ lets us recover \Cref{Imain} as a much simpler special case, since $\R_{>0}^d \cup \{0\} = \B$ is the Boolean semiring, the derivations therefore do not appear at all, and by the assumed $1 > 0$ only spectral points with $\mathbb{K} \in \{\R_+, \TR_+\}$ can occur.

\begin{proof}
	The implication from \ref{alg_ineqs} to \ref{spectral_ineqs} is again simple.
	Indeed if~\ref{alg_ineqs} holds, then we treat the two kinds of spectral points separately:
	\begin{itemize}
		\item For $\phi : S \to \mathbb{K}$ with $\mathbb{K} \in \{\R_+, \R_+^\op, \TR_+, \TR_+^\op\}$ a nondegenerate monotone homomorphism with trivial kernel, we have $\phi(u) > 1$ by power universality of $u$, as well as
			\[
				\phi(x)^n \le \phi(u)^{\lfloor \eps n \rfloor} \phi(y)^n
			\]
			by monotonicity and multiplicativity of $\phi$.
			By taking $n$-th roots and letting $\eps \to 0$, we obtain the desired $\phi(x) \le \phi(y)$.
		\item For $i = 1,\ldots,d$ and $D : S \to \R$ a monotone $\|\cdot\|_i$-derivation with $D(u) = 1$, we have:
			\[
				n D(x) = D(x^n) \le D(u^{\lfloor \eps n \rfloor} y^n) = \lfloor \eps n \rfloor D(u) + D(y^n) = \lfloor \eps n \rfloor + n D(y).
			\]
			Dividing by $n$ and letting $\eps \to 0$ gives the desired $D(x) \le D(y)$.
	\end{itemize}

	For the other implications, note first that the assumed $S / \!\sim{}\! \cong \R_{>0}^d \cup \{0\}$ implies that $S$ is a zerosumfree semidomain, so that \cref{main1} applies.
	Thus if \ref{spectral_ineqs} holds with strict inequalities, then there is nonzero $b \in S$ with $b x \le b y$.
	Choosing any $c \in S$ with $\|c\| = \|b\|^{-1}$ by the surjectivity of $\|\cdot\|$ and taking $a \coloneqq bc$ results in $\|a\| = 1$, and therefore we can achieve
	\[
		a x \le a y, \qquad a \sim 1.
	\]
	We are going to prove~\ref{spectral_ineqs}$\Rightarrow$\ref{alg_ineqs} by the same arguments as in the proof of \Cref{Imain} conducted in Part I, but where the compactness of $\Sper{S}$ now enters through the logarithmic comparison function $\lc_{x,y}$.
	In order to make the present paper self-contained, we spell this out in some detail.

	So suppose that~\ref{spectral_ineqs} holds, not necessarily with strict inequalities.
	Assuming given $\eps > 0$ as in \ref{alg_ineqs}, we choose a positive rational $\ell/m < \eps/3$ and consider $\tilde{x} \coloneqq x^m$ and $\tilde{y} \coloneqq u^\ell y^m$.
	Then using $\tilde{x}$ and $\tilde{y}$ in~\ref{spectral_ineqs} makes the inequalities hold strictly.
	Hence by the previous paragraph, we get $a \in S$ such that $a \tilde{x} \le a \tilde{y}$ and $a \sim 1$.
	Upon chaining inequalities, this implies
	\[
		a \tilde{x}^j \le a \tilde{y}^j
	\]
	for all $j \in \Nplus$. Choosing $k \in \N$ such that $a \le u^k$ and $1 \le a u^k$, plugging in the definitions of $\tilde{x}$ and $\tilde{y}$ gives
	\[
		x^{mj} \le a u^k x^{mj} \le a u^{k + \ell j} y^{mj} \le u^{2 k + \ell j} y^{mj}.
	\]
	By $\ell/m < \eps/3$, this results in
	\[
		x^{mj} \le u^{\lfloor \frac{\eps m j}{2} \rfloor} y^{mj}
	\]
	for all $j \gg 1$, which is an inequality of the desired form. Choosing another $k$ such that $x \le y u^k$, we obtain
	\[
		x^{mj+1} \le u^{\lfloor \frac{\eps m j}{2} \rfloor + k} y^{mj+1} \le u^{\lfloor \eps m j \rfloor} y^{mj+1},
	\]
	where the second inequality holds again for all $j \gg 1$. Fixing such $j$ and weakening a bit further, the two previous steps imply
	\[
		x^{mj} \le u^{\lfloor \eps m j \rfloor} y^{mj}, \qquad x^{mj+1} \le u^{\lfloor \eps (m j + 1) \rfloor} y^{mj+1}.
	\]
	By multiplying powers of these two inequalities, and using that every sufficiently large natural number is a sum of positive integer multiples of $mj$ and $mj+1$, we therefore conclude $x^n \le u^{\lfloor \eps n \rfloor} y^n$ for all $n \gg 1$, as was to be shown for~\ref{alg_ineqs}.

	We now assume strict inequalities in~\ref{spectral_ineqs} and prove the remaining items \ref{main2_asymp}--\ref{main2_catal}. 
	For \ref{main2_asymp}, we use that $\lc_{x,y} : \Sper{S} \to \R_+$ is a continuous real-valued function on a compact Hausdorff space by \Cref{chaus2}. Since this function is strictly positive by assumption, the compactness implies that it is even bounded away from zero by some $\eps > 0$.
	Focusing on spectral points $\phi : S \to \mathbb{K}$ with $\mathbb{K} \in \{\R_+, \TR_+\}$ for the moment, we get
	\[
		\log \frac{\phi(y)}{\phi(x)} \ge \eps \log \phi(u)
	\]
	by the definition of $\lc_{x,y}$ from~\eqref{lev_defn} and $\phi(u) > 1$.
	Upon choosing any positive rational $\ell/m < \eps$, we therefore obtain
	\[
		\phi(x) < \phi(u)^{-\frac{\ell}{m}} \phi(y)
	\]
	for all these $\phi$, or equivalently $\phi(\tilde{x}) < \phi(\tilde{y})$ with $\tilde{x} \coloneqq u^\ell x^m$ and $\tilde{y} \coloneqq y^m$. 
	Arguing in the same way for the other types of spectral points shows that these $\tilde{x}$ and $\tilde{y}$ satisfy all inequalities in~\ref{spectral_ineqs} strictly.
	Applying now \ref{alg_ineqs} with $\eps \coloneqq \ell$, we can find $k,n \in \Nplus$ with $k < \ell n$ and such that $\tilde{x}^n \le u^k \tilde{y}^n$. Equivalently,
	\[
		u^{\ell n} x^{mn} \le u^k y^{mn}.
	\]
	To summarize and reinitialize the symbols in the exponents, we therefore can find large enough $k$ and $m$ such that
	\beq
		\label{asymp_strong}
		u^k x^m \le u^{k-1} y^m.
	\eeq
	This can be weakened to $u^k x^m \le u^k y^m$, showing that the desired inequality holds for \emph{some} exponent.
	On the other hand, let $n$ be such that $x \le y u^n$. Then we obviously have $u^{kn} x^{mn} \le u^{kn} y^{mn}$ by the previous, but also
	\[
		u^{kn} x^{mn + 1} \le u^{(k+1)n} x^{mn} y \le u^{kn} y^{mn + 1},
	\]
	which differs from $u^{kn} x^{mn} \le u^{kn} y^{mn}$ only in having one larger exponent of $x$ and $y$, respectively. Now since every sufficiently large natural number is a sum of positive integer multiples of $mn$ and $mn + 1$, we finally obtain the claim, since the set of all $j \in \N$ for which 
	\[
		u^{kn} x^j \le u^{kn} y^j
	\]
	holds is closed under addition.
	This proves~\ref{main2_asymp}.

	To prove \Cref{main2_power}, we choose $u \coloneqq y$ and use \eqref{asymp_strong} in the form
	\[
		y^k x^m \le y^k y^{m - 1}
	\]
	for suitably large $k$ and $m$. To bring this into a simpler form, note that $k$ can always be increased by multiplying by a power of $y$, while $m$ can be replaced by a multiple of it by chaining and using $y \ge 1$. We can hence assume $m = k$ for simplicity, giving
	\[
		y^k x^k \le y^{2k - 1},
	\]
	which we know to hold for infinitely many $k$. By induction on $j$, we get
	\[
		y^k x^{jk} \le y^{2k + (j-1)(k-1)}
	\]
	for all $j\in\Nplus$. Since $y \ge 1$, we can weaken this to
	\[
		x^{jk} \le y^{2k + (j-1)(k-1)} \le y^{jk},
	\]
	where the second inequality holds for all $j \gg 1$. This proves the desired inequality for some exponent. Choosing $j$ large enough so that the difference of exponents $\ell \coloneqq jk - [2k + (j-1)(k-1)]$ is large enough for $x \le y^\ell$ to be the case, we moreover obtain
	\[
		x^{jk + 1} \le x^{jk} y^\ell \le y^{2k + (j-1)(k-1) + \ell} \le y^{jk + 1}.
	\]
	The claim that the desired inequality holds for all sufficiently large exponents now follows by the same reasoning as for \ref{main2_asymp} in the previous paragraph.

	Finally, we derive the particular form of $a$ in \ref{main2_catal} from \ref{main2_asymp}. The following argument is essentially~\cite[Lemma~5.4]{ocm}.
	The assumed inequality in the form $u^k x^{n+1} \le u^k y^{n+1}$ implies
	\begin{align*}
		\left( u^k \sum_{\ell=0}^n x^\ell y^{n-\ell} \right) x	& = u^k \left( \sum_{\ell=0}^{n-1} x^{\ell+1} y^{n-\ell} + x^{n+1} \right) \\
									& \le u^k \left( y^{n+1} + \sum_{\ell=1}^n x^\ell y^{n+1-\ell} \right) \\
									& = \left( u^k \sum_{\ell=0}^n x^\ell y^{n-\ell} \right) y,
	\end{align*}
	as was to be shown.

\end{proof}

\begin{ex}
	\label{free_ex2}
	Consider the polynomial semiring $\R_+[\underline{X}]$ with the semiring preorder generated by
	\[
		X_1 \ge 1, \quad \ldots, \quad X_d \ge 1,
	\]
	as briefly discussed at the end of \Cref{free_ex}. Since $\R_+[\underline{X}] / \!\sim \mathop{\cong} \R_+$, \Cref{main2} applies with $d = 1$ and proves the following.
	If $f, g \in \R_+[\underline{X}]$ satisfy the same conditions as those listed in \Cref{free_ex}, then we can conclude that for every $n \gg 1$ there is a family of polynomials $(p_\alpha)_{\alpha \in \N^d}$ such that
	\[
		f^n = \sum_{\alpha \in \N^d} p_\alpha, \qquad g^n = \sum_{\alpha \in \N^d} p_\alpha \underline{X}^\alpha.
	\]
	This is by \Cref{main2}\ref{main2_asymp} together with the fact that factors of $u = \prod_i X_i$ can be cancelled from inequalities in this preordered semiring.
\end{ex}

\bibliographystyle{plain}
\bibliography{local_global_principle}

\end{document}